\documentclass{article}
\usepackage{amsmath, amssymb, amsfonts, amsthm, dsfont}
\usepackage{tikz}

\let\originalleft\left
\let\originalright\right
\renewcommand{\left}{\mathopen{}\mathclose\bgroup\originalleft}
\renewcommand{\right}{\aftergroup\egroup\originalright}

\newcommand{\abs}[1]{\left\vert #1 \right\vert}
\newcommand{\ind}[1]{\mathds{1}_{\left\lbrace #1 \right\rbrace}}

\newcommand{\rd}{\mathbb{R}^{d}}

\newcommand{\dd}{\mathrm{d}}

\newcommand{\mminf}{M/M/$\infty$ }

\newcommand{\condexpectation}[2]{\mathbb{E} \left[ #1 \, \middle\vert \, #2 \right]}

\theoremstyle{plain}

\newtheorem{theorem}{Theorem}[section]
\newtheorem{corollary}[theorem]{Corollary}
\newtheorem{lemma}[theorem]{Lemma}
\newtheorem{proposition}[theorem]{Proposition}

\theoremstyle{definition}

\newtheorem{definition}[theorem]{Definition}
\newtheorem{example}[theorem]{Example}

\newtheorem*{acknowledgement}{Acknowledgement}

\newcommand{\closedball}[2]{B \left[ #1 , #2 \right]}
\newcommand{\closedballplus}[2]{B_{+} \left[ #1 , #2 \right]}

\newcommand{\openball}[2]{B \left( #1 , #2 \right)}
\newcommand{\openballplus}[2]{B_{+} \left( #1 , #2 \right)}

\title{A large deviations principle for infinite-server queues in a random environment}
\author{H. M. Jansen$^{1,2}$, M. R. H. Mandjes$^{1}$, K. De Turck$^{2}$, S. Wittevrongel$^{2}$}

\begin{document}
\maketitle

\begin{abstract}
\noindent This paper studies an infinite-server queue in a random environment, meaning that the arrival rate, the service requirements and the server work rate are modulated by a general c\`{a}dl\`{a}g stochastic background process. To prove a large deviations principle, the concept of attainable parameters is introduced. Scaling both the arrival rates and the background process, a large deviations principle for the number of jobs in the system is derived using attainable parameters. Finally, some known results about Markov-modulated infinite-server queues are generalized and new results for several background processes and scalings are established in examples.
\end{abstract}

\newline

\noindent $^{1}$ Korteweg-de Vries Institute for Mathematics,
University of Amsterdam, Science Park 904, 1098 XH Amsterdam, the Netherlands.
\newline

\noindent $^{2}$ TELIN, Ghent University, Sint-Pietersnieuwstraat 41,
B-9000 Ghent, Belgium.
\newline

\noindent {\it E-mail}. {\tt\{h.m.jansen|m.r.h.mandjes\}@uva.nl}, {\tt\{kdeturck|sw\}@telin.ugent.be}

\section{Introduction}
\label{sec:intro}
The infinite-server queue is one of the fundamental models in queueing theory. Its distinguishing feature is the presence of an infinite number of servers, so that jobs are served independently and there are no waiting times. This leads to explicit formulas for many quantities of interest, especially for M/M/$\infty$ queues, where jobs arrive according to a Poisson process and the service requirements have an exponential distribution. In practice, however, one often observes time-varying arrival intensities, service requirement distributions and server work rates. This calls for adequate modeling.

A natural way to incorporate time-dependence is to consider an M/M/$\infty$ queue in a random environment. In this case there is an independent background process that modulates the arrival rate, the service requirement distribution and the work rate of the servers.

\emph{Model.} In this paper, we study the case where the background process is a general stochastic process $J$ whose paths are right-continuous and have finite left limits, i.e.\ $J$ has c\`{a}dl\`{a}g paths. The process $J$ modulates the arrival rate, the service requirement distribution and the server work rate in the following way. When $J$ is in state $x$, jobs arrive according to a Poisson process with intensity $\lambda \left( x \right)$. Upon arrival, a job draws an independent service requirement from an exponential distribution with parameter $\kappa \left( x \right)$ if $J$ is in state $x$ when the job arrives. Then the service requirement of the job is processed by a server, whose work rate is $\mu \left( x \right)$ while $J$ is in state $x$. Immediately after its service requirement has been processed, a job leaves the system.

\emph{Main result.} The main result of this paper is a full large deviations principle (LDP) for the transient number of jobs in the system, under a scaling of the arrival rate and the background process. To arrive at this result, we first show that the number of jobs in the system at time $t \geq 0$ has a Poisson distribution with random parameter $\phi_{t} \left( J \right)$. Then we scale $\lambda \mapsto n \lambda$ and we scale $J \mapsto J_{n}$ such that the normalized random parameter $\phi_{t} \left( J_{n} \right)$ satisfies an LDP. Under this scaling, we derive the LDP for the transient number of jobs in the system.

\emph{Literature.} The amount of literature on infinite-server queues in a random environment is quite small. Moreover, almost all papers on this topic (with notable exception \cite{BM2013}) study Model~I or Model~II (cf.\ \cite{BDKM2014}). In both models, jobs arrive according to a Poisson process with intensity $\lambda \left( x \right)$ when the background process is in state $x$. In Model~I, service requirements have a standard exponential distribution and servers work at rate $\mu \left( x \right)$ when the background process is in state $x$. This is equivalent to the jobs being subject to a modulated hazard rate. In Model~II, service requirements have an exponential distribution with parameter $\kappa \left( x \right)$ when the background process is in state $x$ and servers work at constant rate $1$.

An early reference is \cite{OP1986}, which analyzes Model~I when the background process is a continuous-time Markov chain. Important results in \cite{OP1986} are a recursion for the factorial moments of the number of jobs and the observation that the steady-state distribution is not of some `matrix-Poisson' type.

Other important results can be found in \cite{Dauria2008}, which studies Model~I when the background process is a semi-Markov process with finite state space. The crucial observation in \cite{Dauria2008} is that the stationary number of jobs has a Poisson distribution with a random parameter that is determined by the background process. Moreover, the factorial moments of the number of jobs are computed via a recursion. These results are generalized in \cite{FA2009}.

The observation in \cite{Dauria2008} is used to obtain time-scaling results in both the central limit regime and the large deviations regime. In the central limit regime, \cite{BKMT2014} and \cite{BDM2014} derive central limit theorems for Markov-modulated infinite-server queues for several models and scalings. In this regime, the so-called deviation matrix (cf.\ \cite{CV2002}) plays an important role. In the large deviations regime, \cite{BDKM2014} and \cite{BM2013} compute optimal paths to obtain rate functions under a linear scaling of the arrival rates, given that the background process is an irreducible continuous-time Markov chain. The former studies Model~I, whereas the latter studies Model~II for a class of service requirement distributions that includes the exponential distribution.

As mentioned, we show that the number of jobs in the system has a Poisson distribution with a random parameter, which can be interpreted as a mixture of Poisson distributions. In \cite{Biggins2004}, an LDP is derived for mixtures that satisfy certain assumptions. However, apart from the assumption that the normalized random parameter satisfies an LDP, these assumptions are either superfluous or too restrictive in our case. In particular, we do not assume that the sequence of measures induced by the normalized random parameter is exponentially tight, so we cannot use the arguments in \cite{Biggins2004}. Hence, we need a different approach to obtain an LDP.

\emph{Contributions.} In more detail, the contributions of this paper are the following. We generalize known models by considering a general c\`{a}dl\`{a}g background process instead of a semi-Markov background process with finite state space. Moreover, in our model the background process modulates both the service requirement distributions and the server work rate, whereas previous papers considered models in which either the service requirement distributions or the server work rate was modulated. In particular, our model generalizes Model~I and Model~II.

Using elementary arguments, we show that in this model the transient number of jobs has a Poisson distribution with random parameter. We scale the arrival rate linearly and we scale the background process such that the normalized random parameter satisfies an LDP. Under this scaling, we obtain a full LDP for the number of jobs in the system. To the best of our knowledge, this is the first time that a full LDP is presented for modulated infinite-server queues. To prove the LDP, we introduce the concept of attainable parameters and use a variation on Varadhan's~Lemma. These tools enable us to avoid the assumptions in \cite{Biggins2004}.

The theory is illustrated by examples that show rate functions that cannot be obtained via background processes with finite state space. Additionally, we show that completely different background processes may lead to the same LDP, even in highly nontrivial cases.

\emph{Organization.} The rest of this paper is organized as follows. In Section~2, we describe the model and provide some of its basic properties. Additionally, we fix some notation. In Section~3, we introduce the concept of attainable parameters and prove an LDP for the number of jobs in the system. In Section~4, we show that the rate function corresponding to this LDP has a simple description when we do not scale the background process. As an illustration, we work out some examples. In Section~5, we give examples in which we do scale the background process. In Section~6, we briefly discuss the results and point out some topics for future research. The appendices provide some technical details about the number of jobs in the system (Section~A), continuity in Skorokhod space (Section~B) and properties of Poisson random variables (Section~C).

\section{Model and problem description}
\label{sec:modprob}
We study an infinite-server queue with modulated arrival rates, service requirements and server work rates. The precise mathematical setup of the model and some of its basic properties are provided in Section~\ref{sec:transient}. In words, the model may be described as follows.

Let $\left( J \left( t \right) \right)_{t \geq 0}$ be a c\`{a}dl\`{a}g stochastic process with state space $\mathcal{E}$, which is assumed to be a metric space. We will refer to the process $J$ as the background process or modulating process. While the background process is in state $x \in \mathcal{E}$, jobs enter the system following a Poisson process with intensity $\lambda \left( x \right) \geq 0$. 

When job $k$ enters the system, it draws a service requirement from an independent exponential distribution with parameter $\kappa \left( y \right)$ if the background process is in state $y \in \mathcal{E}$ upon its arrival. Server $k$ processes this service requirement at rate $\mu \left( z \right)$ while the background process is in state $z \in \mathcal{E}$. Job $k$ leaves the system when its service requirement has been processed.

We denote a modulated infinite-server queue by the quadruple $\left( J , \lambda , \kappa , \mu \right)$. Additionally, we denote the number of jobs in this system at time $t$ by $M \left( t \right)$. In Section~\ref{sec:transient} it is shown that $M \left( t \right)$ has a Poisson distribution with random parameter
\begin{align}
\phi_{t} \left( J \right) = \int_{0}^{t} \lambda \left( J \left( s \right) \right) e^{- \kappa \left( J \left( s \right) \right) \int_{s}^{t} \mu \left( J \left( r \right) \right) \, \dd r} \, \dd s. \label{eq:randomparameter}
\end{align}
This will turn out to be a crucial property in this paper.

We are interested in events with an unusual number of jobs in the system. More precisely, we would like to prove an LDP for the number of jobs in the system. A sequence of probability measures $\left\lbrace \tau_{n} \right\rbrace_{n \in \mathbb{N}}$ is said to satisfy an LDP with rate function $\rho$ if there exists a lower semi-continuous function $\rho \colon \mathcal{X} \to \left[ 0,\infty \right]$ such that
\begin{align*}
\limsup_{n \to \infty} \frac{1}{n} \log \tau_{n} \left( F \right) &\leq - \inf_{a \in F} \rho \left( a \right)
\intertext{for all closed sets $F$ and}
\liminf_{n \to \infty} \frac{1}{n} \log \tau_{n} \left( G \right) &\geq - \inf_{a \in G} \rho \left( a \right)
\end{align*}
for all open sets $G$, where each $\tau_{n}$ is defined on the Borel $\sigma$-algebra of the topological space $\mathcal{X}$. A sequence of random variables is said to satisfy an LDP with rate function $\rho$ if the sequence of measures induced by the random variables satisfies an LDP with rate function $\rho$. Importantly, we do not assume that $\rho$ is a good rate function, i.e., we do not assume that $\rho$ has compact level sets.

As mentioned, we would like to prove an LDP for the number of jobs in the system. To analyze this problem, we will scale the arrival rates via $\lambda \left( x \right) \mapsto n \lambda \left( x \right)$, i.e., we linearly speed up the arrivals. In addition, we will scale the background process via $J \mapsto J_{n}$. 
Formally, scaling $\lambda \left( x \right) \mapsto n \lambda \left( x \right)$ and $J \mapsto J_{n}$ means that we start with an infinite-server queue $\left( J , \lambda , \kappa , \mu \right)$ and then consider the sequence of infinite-server queues $\left\lbrace \left( J_{n} , n\lambda , \kappa , \mu \right) \right\rbrace_{n \in \mathbb{N}}$.

Given the scalings $\lambda \left( x \right) \mapsto n \lambda \left( x \right)$ and $J \mapsto J_{n}$, we denote the corresponding number of jobs in the system by $M_{n} \left( t \right)$. It follows immediately from equation~\eqref{eq:randomparameter} that $M_{n} \left( t \right)$ has a Poisson distribution with random parameter
\begin{align*}
n \phi_{t} \left( J_{n} \right) = \int_{0}^{t} n \lambda \left( J_{n} \left( s \right) \right) e^{- \kappa \left( J_{n} \left( s \right) \right) \int_{s}^{t} \mu \left( J_{n} \left( r \right) \right) \, \dd r} \, \dd s.
\end{align*}
The normalized random parameter $\phi_{t} \left( J_{n} \right)$ induces a sequence of probability measures $\left\lbrace \nu_{n} \right\rbrace_{n \in \mathbb{N}}$ on $\mathbb{R}$ via $\nu_{n} \left( B \right) = \mathbb{P} \left( \phi_{t} \left( J_{n} \right) \in B \right)$ for Borel sets $B \subset \mathbb{R}$.

We will assume that the sequence of probability measures $\left\lbrace \nu_{n} \right\rbrace_{n \in \mathbb{N}}$ satisfies an LDP with rate function $\psi$. 
Note that $\left\lbrace \nu_{n} \right\rbrace_{n \in \mathbb{N}}$ trivially satisfies an LDP when $\nu_{n} = \nu_{n+1}$ for all $n \in \mathbb{N}$, so this assumption covers the case in which the background process is not scaled.

Given the scaling, we denote the number of jobs in the system at time $t$ by $M_{n} \left( t \right)$ and consider the normalized random variable $\tfrac{1}{n} M_{n} \left( t \right)$. Our goal is to prove an LDP for $\tfrac{1}{n} M_{n} \left( t \right)$ and to describe the corresponding rate function.

Throughout this paper, we will also use the following notation. We denote the closure of a set $A$ by $\mathrm{cl} A$. We write $\openball{x}{\epsilon}$ for the open ball with center $x \in \rd$ and radius $\epsilon > 0$ and $\closedball{x}{\epsilon}$ for its closure. For notational convenience, we will sometimes write $\mathbb{R}_{+}$ for $\left[ 0,\infty \right)$, $\openballplus{x}{\epsilon}$ for $\openball{x}{\epsilon} \cap \mathbb{R}_{+}$ and $\closedballplus{x}{\epsilon}$ for $\closedball{x}{\epsilon} \cap \mathbb{R}_{+}$.
As is customary, we define $\exp \left( - \infty \right) = 0$ and $\log \left( 0 \right) = -\infty$.

\section{A large deviations principle}
In this section we will prove an LDP for the number of jobs in the system under a scaling of the arrival rates and the background process, i.e., we will prove an LDP for $\frac{1}{n} M_{n} \left( t \right)$. It will turn out that so-called attainable parameters determine the rate function corresponding to the LDP.
\begin{definition}
Given a scaling $J \mapsto J_{n}$, a real number $\gamma \in [ 0,\infty )$ is called an \emph{attainable parameter} at time $t \geq 0$ if for all $\epsilon > 0$ there exists $N_{\epsilon} \in \mathbb{N}$ such that $\mathbb{P} \left( \phi_{t} \left( J_{n} \right) \in \openball{\gamma}{\epsilon} \right) = \nu_{n} \left(  \openball{\gamma}{\epsilon} \right) > 0$ for all $n \geq N_{\epsilon}$. The set of all attainable parameters at time $t$ is denoted by $\mathcal{R} \left( t \right)$.
\end{definition}

The intuition behind attainable parameters is as follows. The number of jobs in the system has a Poisson distribution with a random parameter that is completely determined by the background process. Basically, the background process samples the Poisson parameter. A real number $\gamma$ is an attainable parameter if, for all $n$ large enough, the scaled background process samples parameters close to $\gamma$ with positive probability.

As mentioned before, we will prove an LDP for $\frac{1}{n} M_{n} \left( t \right)$ by scaling $\lambda \left( x \right) \mapsto n \lambda \left( x \right)$ and $J \mapsto J_{n}$ such that the sequence of probability measures $\left\lbrace \nu_{n} \right\rbrace_{n \in \mathbb{N}}$ satisfies an LDP with rate function $\psi$. The rate function $I \colon \mathbb{R} \to \left[ 0,\infty \right]$ governing the LDP for $\frac{1}{n} M_{n} \left( t \right)$ is given by
\begin{align}
I \left( a \right) = \inf_{\gamma \in \mathcal{R} \left( t \right)} \left[ \ell \left( \gamma ; a \right) + \psi \left( \gamma \right) \right], \label{eq:Idef}
\end{align}
where $\ell \left( \gamma ; \cdot \right)$ is the Fenchel-Legendre transform of the Poisson cumulant generating function with parameter $\gamma$. It will turn out (cf.\ Lemma~\ref{lem:Rnonempty}) that
\begin{align}
I \left( a \right) = \inf_{\gamma \in \mathcal{R} \left( t \right)} \left[ \ell \left( \gamma ; a \right) + \psi \left( \gamma \right) \right] = \inf_{\gamma \in \left\lbrace \psi < \infty \right\rbrace} \left[ \ell \left( \gamma ; a \right) + \psi \left( \gamma \right) \right]. \label{eq:Idef2}
\end{align}
However, we will take the infimum over $\mathcal{R} \left( t \right)$ rather than over $\left\lbrace \psi < \infty \right\rbrace$ to stress that attainability of parameters is the crucial property for proving the LDP.

Before we can give the proof, we have to settle some technical details.
First, it is not immediately clear whether the function $I$ is indeed a rate function or even whether $I$ is well defined. In particular, it is not clear whether $\mathcal{R} \left( t \right)$ is a non-empty set. However, the assumption that the sequence $\left\lbrace \nu_{n} \right\rbrace_{n \in \mathbb{N}}$ satisfies an LDP implies that $\mathcal{R} \left( t \right)$ is non-empty, as the following lemma shows.
\begin{lemma}
\label{lem:Rnonempty}
Let the scaling $J \mapsto J_{n}$ be such that $\left\lbrace \nu_{n} \right\rbrace_{n \in \mathbb{N}}$ satisfies an LDP with rate function $\psi$. Then $\mathcal{R} \left( t \right)$ is a non-empty closed subset of $\left[ 0,\infty \right)$ and $\left\lbrace \psi < \infty \right\rbrace \subset \mathcal{R} \left( t \right)$.
\end{lemma}
\begin{proof}
Suppose that $\gamma \in \mathbb{R} \setminus \mathcal{R} \left( t \right)$. Then there exists $\epsilon > 0$ such that for all $n \in \mathbb{N}$ there exists $k_{n} \in \mathbb{N}$ such that $k_{n} \geq n$ and $\nu_{k_{n}} \left( \openball{\gamma}{\epsilon} \right) = 0$. This implies that $\openball{\gamma}{\epsilon} \subset \mathbb{R} \setminus \mathcal{R} \left( t \right)$, so $\mathcal{R} \left( t \right)$ is closed.
Moreover, we must have
\begin{align*}
\liminf_{n \to \infty} \frac{1}{n} \log \nu_{n} \left( \openball{\gamma}{\epsilon} \right) = -\infty = - \inf_{a \in \openball{\gamma}{\epsilon}} \psi \left( a \right),
\end{align*}
so $\psi \left( a \right) = \infty$ for all $a \in \openball{\gamma}{\epsilon}$. Then $\mathbb{R} \setminus \mathcal{R} \left( t \right) \subset \left\lbrace \psi = \infty \right\rbrace$ and $\left\lbrace \psi < \infty \right\rbrace \subset \mathcal{R} \left( t \right)$. The fact that $\psi$ is a rate function implies that $\left\lbrace \psi < \infty \right\rbrace$ is non-empty. The statement of the lemma follows immediately.
\end{proof}
From the previous lemma it follows that $I$ is a well defined function. The fact that $I$ is a rate function is implied by Proposition~\ref{pr:Iratefunction} and the functions $\ell$ and $\psi$ being rate functions.

The next lemma is a variation on Varadhan's Lemma. Contrary to Varadhan's Lemma, it does not require that a given function $f$ is continuous. Instead, it requires that a weaker condition is fulfilled. We will use this lemma to obtain the large deviations upper bound, by applying it to functions $f$ of the form described in Proposition~\ref{pr:verifyvaradhan}.
\begin{lemma}
\label{lem:varvar}
Let $\mathcal{X}$ be a topological space and let $\left\lbrace \xi_{n} \right\rbrace_{n \in \mathbb{N}}$ be a sequence of measures defined on its Borel $\sigma$-algebra. Suppose that $\left\lbrace \xi_{n} \right\rbrace_{n \in \mathbb{N}}$ satisfies an LDP with rate function $\varrho$. Let $f \colon \mathcal{X} \to \left[ -\infty , 0 \right]$ be a Borel measurable function such that $f^{-1} \left( \left[ a,b \right] \right)$ is a closed set for all $a,b \in \left( -\infty , 0 \right]$ satisfying
\begin{align*}
\sup_{x \in \mathcal{X}} \left[ f \left( x \right) - \varrho \left( x \right) \right] \leq a \leq b \leq 0.
\end{align*}
Then it holds that
\begin{align*}
\limsup_{n \to \infty} \frac{1}{n} \log \int_{\mathcal{X}} e^{n f \left( x \right)} \xi_{n} \left( \dd x \right) \leq \sup_{x \in \mathcal{X}} \left[ f \left( x \right) - \varrho \left( x \right) \right].
\end{align*}
\end{lemma}
\begin{proof}
This follows immediately from \cite[Lem.~2.2]{JMDW2014}
\end{proof}

With these technical details settled, we can prove the following LDP. Its proof exploits two elementary observations. First, conditional on a value of the random parameter $\phi_{t} \left( J_{n} \right)$, the number of jobs in the system has the same distribution as the number of jobs in the system in the \mminf setting. Second, the number of jobs in the system in the \mminf setting has the same distribution as a sum of i.i.d.\ Poisson random variables. Combined with some analytical results, these observations enable us to prove the LDP.

\begin{theorem}
\label{thm:LDP}
Consider a modulated infinite-server queue $\left( J , \lambda , \kappa , \mu \right)$ as described in Section~\ref{sec:modprob}. Scale $\lambda \left( x \right) \mapsto n \lambda \left( x \right)$ and $J \mapsto J_{n}$ such that $\left\lbrace \nu_{n} \right\rbrace_{n \in \mathbb{N}}$ satisfies an LDP with rate function $\psi$. Then the rescaled number of jobs in the system $\frac{1}{n} M_{n} \left ( t \right)$ satisfies an LDP with rate function $I$ as defined in equation~\eqref{eq:Idef}, so
\begin{align}
\limsup_{n \to \infty} \frac{1}{n} \log \mathbb{P} \left( \frac{1}{n} M_{n} \left ( t \right) \in F \right) &\leq - \inf_{a \in F} I \left( a \right) \label{eq:upperbound}
\intertext{for any closed set $F \subset \mathbb{R}$ and}
\liminf_{n \to \infty} \frac{1}{n} \log \mathbb{P} \left( \frac{1}{n} M_{n} \left ( t \right) \in G \right) &\geq - \inf_{a \in G} I \left( a \right) \label{eq:lowerbound}
\end{align}
for any open set $G \subset \mathbb{R}$.
\end{theorem}
\begin{proof}
For $\lambda \geq 0$, let $P_{0} \left( \lambda \right) , P_{1} \left( \lambda \right), P_{2} \left( \lambda \right) , \ldots$ denote a sequence of i.i.d.\ random variables that have a Poisson distribution with parameter $\lambda$. Let $F \subset \mathbb{R}$ be a closed set and let $G \subset \mathbb{R}$ be an open set.

To prove the upper~bound~\eqref{eq:upperbound}, recall that $M_{n} \left( t \right)$ has a Poisson distribution with random parameter $n \phi_{t} \left( J_{n} \right)$. Then we may write
\begin{align*}
\limsup_{n \to \infty} \frac{1}{n} \log \mathbb{P} \left( \frac{1}{n} M_{n} \left( t \right) \in F \right)
&= \limsup_{n \to \infty} \frac{1}{n} \log \int_{\left[ 0,\infty \right)} \mathbb{P} \left( \frac{1}{n} P_{0} \left( n \gamma \right) \in F \right) \, \nu_{n} \left( \dd \gamma \right) \\
&= \limsup_{n \to \infty} \frac{1}{n} \log \int_{\left[ 0,\infty \right)} \mathbb{P} \left( \frac{1}{n} \sum_{i=1}^{n} P_{i} \left( \gamma \right) \in F \right) \, \nu_{n} \left( \dd \gamma \right) \\
&\leq \limsup_{n \to \infty} \frac{1}{n} \log \int_{\left[ 0,\infty \right)} 2 e^{n \left[ -\inf_{a \in F} \ell \left( \gamma ; a \right) \right]} \, \nu_{n} \left( \dd \gamma \right)\\
&= \limsup_{n \to \infty} \frac{1}{n} \log \int_{\left[ 0,\infty \right)} e^{n \left[ -\inf_{a \in F} \ell \left( \gamma ; a \right) \right]} \, \nu_{n} \left( \dd \gamma \right).
\end{align*}
The inequality above is an immediate result of the proof of Cram\'{e}r's~Theorem in $\mathbb{R}$ as provided in \cite{dz1998}.

According to Proposition~\ref{pr:verifyvaradhan}, the function $\gamma \mapsto -\inf_{a \in F} \ell \left( \gamma ; a \right)$ satisfies the assumptions of Lemma~\ref{lem:varvar}. Moreover, $\left\lbrace \nu_{n} \right\rbrace_{n \in \mathbb{N}}$ satisfies an LDP both in $\mathbb{R}$ and in $\left[ 0,\infty \right)$ with rate function $\psi$ (cf.\ \cite[Lem.~4.1.5]{dz1998}). Hence, we may apply Lemma~\ref{lem:varvar} to obtain
\begin{align*}
\limsup_{n \to \infty} \frac{1}{n} \log \mathbb{P} \left( \frac{1}{n} M_{n} \left( t \right) \in F \right)
&\leq \limsup_{n \to \infty} \frac{1}{n} \log \int_{\left[ 0,\infty \right)} e^{n \left[ -\inf_{a \in F} \ell \left( \gamma ; a \right) \right]} \, \nu_{n} \left( \dd \gamma \right)\\
&\leq \sup_{\gamma \in \left[ 0,\infty \right)} \left[ -\inf_{a \in F} \ell \left( \gamma ; a \right) - \psi \left( \gamma \right) \right]\\
&= -\inf_{a \in F} \inf_{\gamma \in \left[ 0,\infty \right)} \left[ \ell \left( \gamma ; a \right) + \psi \left( \gamma \right) \right]\\
&= -\inf_{a \in F} \inf_{\gamma \in \mathcal{R} \left( t \right)} \left[ \ell \left( \gamma ; a \right) + \psi \left( \gamma \right) \right]\\
&= -\inf_{a \in F} I \left( a \right).
\end{align*}
The fact that we only have to consider the infimum over $\mathcal{R} \left( t \right)$ follows from Lemma~\ref{lem:Rnonempty}. This proves the upper bound.

To prove the lower~bound~\eqref{eq:lowerbound}, let $\lambda \in \mathcal{R} \left( t \right)$ and $\epsilon > 0$. Define $\lambda^{-}_{\epsilon} = \max \left\lbrace 0 , \lambda - \epsilon \right\rbrace$ and $\lambda^{+}_{\epsilon} = \lambda + \epsilon$. By definition of the set $\mathcal{R} \left( t \right)$ there exists $N_{\epsilon}$ such that $\mathbb{P} \left( \phi_{t} \left( J_{n} \right) \in \openball{\lambda}{\epsilon} \right) > 0$ for all $n \geq N_{\epsilon}$.

Fix $x \in G$. Because $G$ is open, there exists $\delta > 0$ such that $\openball{x}{\delta} \subset G$. Observe that
\begin{align*}
\mathbb{P} \left( \frac{1}{n} M_{n} \left( t \right) \in G \right) 
&\geq \mathbb{P} \left( \frac{1}{n} M_{n} \left( t \right) \in \openball{x}{\delta} \right) \\
&\geq \mathbb{P} \left( \frac{1}{n} M_{n} \left( t \right) \in \openball{x}{\delta} \, ; \, \phi_{t} \left( J_{n} \right) \in \openball{\lambda}{\epsilon} \right) \\
&= \mathbb{P} \left( \frac{1}{n} M_{n} \left( t \right) \in \openball{x}{\delta} \, \middle\vert \, \phi_{t} \left( J_{n} \right) \in \openball{\lambda}{\epsilon} \right) \mathbb{P} \left( \phi_{t} \left( J_{n} \right) \in \openball{\lambda}{\epsilon} \right)
\end{align*}
for all $n \geq N_{\epsilon}$, where the equality follows from the fact that $\mathbb{P} \left( \phi_{t} \left( J_{n} \right) \in \openball{\lambda}{\epsilon} \right) > 0$ for all $n \geq N_{\epsilon}$. Then we get
\begin{align*}
\liminf_{n \to \infty} \frac{1}{n} \log \mathbb{P} \left( \frac{1}{n} M_{n} \left( t \right) \in G \right) &\geq \\
\liminf_{n \to \infty} \frac{1}{n} \log \mathbb{P} \left( \frac{1}{n} M_{n} \left( t \right) \in \openball{x}{\delta} \, \middle\vert \, \phi_{t} \left( J_{n} \right) \in \openball{\lambda}{\epsilon} \right) \\
+ \liminf_{n \to \infty} \frac{1}{n} \log \mathbb{P} \left( \phi_{t} \left( J_{n} \right) \in \openball{\lambda}{\epsilon} \right).
\end{align*}
Recall that $\phi_{t} \left( J_{n} \right)$ satisfies an LDP with rate function $\psi$, so
\begin{align*}
\liminf_{n \to \infty} \frac{1}{n} \log \mathbb{P} \left( \phi_{t} \left( J_{n} \right) \in \openball{\lambda}{\epsilon} \right) \geq - \inf_{a \in \openball{\lambda}{\epsilon}} \psi \left( a \right)
\end{align*}
by assumption. Moreover, it holds that
\begin{align*}
\liminf_{n \to \infty} \frac{1}{n} \log \mathbb{P} \left( \frac{1}{n} M_{n} \left( t \right) \in \openball{x}{\delta} \, \middle\vert \, \phi_{t} \left( J_{n} \right) \in \openball{\lambda}{\epsilon} \right) &= \\
\liminf_{n \to \infty} \frac{1}{n} \log \mathbb{P} \left( \frac{1}{n} M_{n} \left( t \right) \in \openball{x}{\delta} \, \middle\vert \, \phi_{t} \left( J_{n} \right) \in \openball{\lambda}{\epsilon} \cap \mathbb{R}_{+} \right) &\geq \\
\liminf_{n \to \infty} \inf_{\gamma \in \openball{\lambda}{\epsilon} \cap \mathbb{R}_{+}} \frac{1}{n} \log \mathbb{P} \left( \frac{1}{n} \sum_{i = 1}^{n} P_{i} \left( \gamma \right) \in \openball{x}{\delta} \right) &= \\
\min_{\gamma \in \left\lbrace \lambda^{-}_{\epsilon} , \lambda^{+}_{\epsilon} \right\rbrace} \left[ - \inf_{a \in \openball{x}{\delta}} \ell \left( \gamma ; a \right) \right].
\end{align*}
The equality above is established in Proposition~\ref{pr:largedevopenint}. Combining the results, we obtain that
\begin{align*}
\mathbb{P} \left( \frac{1}{n} M_{n} \left( t \right) \in G \right)
\geq \min_{\gamma \in \left\lbrace \lambda^{-}_{\epsilon} , \lambda^{+}_{\epsilon} \right\rbrace} \left[ - \inf_{a \in \openball{x}{\delta}} \ell \left( \gamma ; a \right) \right]  - \inf_{a \in \openball{\lambda}{\epsilon}} \psi \left( a \right).
\end{align*}
This holds for all $\epsilon > 0$ and small enough $\delta > 0$. Taking limits, we get
\begin{align*}
\lim_{\epsilon \downarrow 0} \min_{\gamma \in \left\lbrace \lambda^{-}_{\epsilon} , \lambda^{+}_{\epsilon} \right\rbrace} \left[ - \inf_{a \in \openball{x}{\delta}} \ell \left( \gamma ; a \right) \right] = - \inf_{a \in \openball{x}{\delta}} \ell \left( \lambda ; a \right)
\end{align*}
thanks to Proposition~\ref{pr:verifyvaradhan} and
\begin{align*}
\lim_{\epsilon \downarrow 0} \inf_{a \in \openball{\lambda}{\epsilon}} \psi \left( a \right) = \psi \left( \lambda \right),
\end{align*}
because $\psi$ is lower semi-continuous. Similarly, we get $\lim_{\delta \downarrow 0} \inf_{a \in \openball{x}{\delta}} \ell \left( \lambda ; a \right) = \ell \left( \lambda ; x \right)$. Hence, it follows that
\begin{align*}
\mathbb{P} \left( \frac{1}{n} M_{n} \left( t \right) \in G \right)
&\geq \lim_{\delta \downarrow 0} \lim_{\epsilon \downarrow 0} \left[ \min_{\gamma \in \left\lbrace \lambda^{-}_{\epsilon} , \lambda^{+}_{\epsilon} \right\rbrace} \left[ - \inf_{a \in \openball{x}{\delta}} \ell \left( \gamma ; a \right) \right]  - \inf_{a \in \openball{\lambda}{\epsilon}} \psi \left( a \right) \right]\\
&= - \left[ \ell \left( \lambda ; x \right) + \psi \left( \lambda \right) \right].
\end{align*}
Since $x \in G$ and $\lambda \in \mathcal{R} \left( t \right)$ were arbitrary, we obtain
\begin{align*}
\mathbb{P} \left( \frac{1}{n} M_{n} \left( t \right) \in G \right)
&\geq \sup_{a \in G} \sup_{\lambda \in \mathcal{R} \left( t \right)} \left[ - \left[ \ell \left( \lambda ; x \right) + \psi \left( \lambda \right) \right] \right]\\
&= - \inf_{a \in G} I \left( a \right),
\end{align*}
which completes the proof.
\end{proof}
The proof of Theorem~\ref{thm:LDP} contains familiar elements. First, the upper bound is proved using a Chernoff bound combined with a variation on Varadhan's Lemma. Second, the lower bound is proved by considering `the most likely of all unlikely scenarios', which is similar to the method used in \cite{BDKM2014} and \cite{BM2013}. However, the proofs there relied on properties of irreducible continuous-time Markov chains and the computation of optimal paths, whereas we consider general c\`{a}dl\`{a}g background processes via attainable parameters.

\section{Examples: unscaled background processes}
\label{sec:exunscaled}
Given the scaling $\lambda \mapsto n \lambda$ and $J \mapsto J_{n}$, Theorem~\ref{thm:LDP} provides a full LDP for $\frac{1}{n} M_{n} \left( t \right)$ and describes the corresponding rate function. In the upcoming examples we will consider cases in which the background process is not scaled and we will use Theorem~\ref{thm:LDP} to verify or extend known results and to obtain new results.

Throughout this section we will assume that the background process is not scaled, i.e., $J_{n} = J$ for all $n \in \mathbb{N}$ for some c\`{a}dl\`{a}g stochastic process $J$. The following lemma is trivial, but plays a central role in this section.
\begin{lemma}
\label{lem:unscaledldp}
If $J_{n} = J$ for all $n \in \mathbb{N}$, then the sequence $\left\lbrace \phi_{t} \left( J_{n} \right) \right\rbrace_{n \in \mathbb{N}}$ satisfies an LDP with rate function $\psi$. In this case it holds that $\mathcal{R} \left( t \right) = \left\lbrace \psi < \infty \right\rbrace = \left\lbrace \psi = 0 \right\rbrace$.
\end{lemma}
Hence, when the background process is not scaled, we have the special property that $\mathcal{R} \left( t \right) = \left\lbrace \psi = 0 \right\rbrace$. This will enable us to compute explicit rate functions in the examples. In these computations, we will extensively use the following properties of the rate function $I$ and properties of step functions in Skorokhod space.

Recall that the rate function $I$ is given by
\begin{align*}
I \left( a \right) = \inf_{\gamma \in \mathcal{R} \left( t \right)} \left[ \ell \left( \gamma ; a \right) + \psi \left( \gamma \right) \right],
\end{align*}
and that $\mathcal{R} \left( t \right) = \left\lbrace \psi = 0 \right\rbrace$ (see Lemma~\ref{lem:unscaledldp}). Hence, we get
\begin{align}
I \left( a \right) = \inf_{\gamma \in \mathcal{R} \left( t \right)} \ell \left( \gamma ; a \right). \label{eq:Iunscaled}
\end{align}
In this case, we can give a simpler and more explicit description of $I$, using the following properties of the function $\ell$.

For $\gamma \geq 0$, the function $\ell \left( \gamma ; \cdot \right)$ is the Fenchel-Legendre transform of the Poisson cumulant generating function with parameter $\gamma$ and is given by
\begin{align}
\ell \left( \gamma ; a \right) = \begin{cases} \infty & a < 0;\\ \gamma & a = 0;\\ \gamma - a + a \log \left( a / \gamma \right) & a > 0. \end{cases} \label{eq:poisratefunc}
\end{align}
For $\gamma = 0$ and $a > 0$, we understand that $\gamma - a + a \log \left( a / \gamma \right) = \infty$. An important observation is that the following inequalities hold for $0 \leq \gamma_{1} \leq \gamma_{2} < \infty$:
\begin{align}
\ell \left( \gamma_{1} ; a \right) \leq \ell \left( \gamma_{2} ; a \right) & & & \forall a \in \left[ 0,\gamma_{1} \right]; \label{eq:ell_leq} \\
\ell \left( \gamma_{1} ; a \right) \geq \ell \left( \gamma_{2} ; a \right) & & & \forall a \in \left[ \gamma_{2},\infty \right). \label{eq:ell_geq}
\end{align}
See Figure~\ref{fig:graphell} for an illustration.

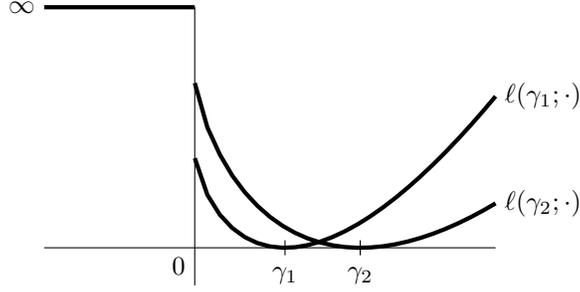
\begin{figure}
\centering
\begin{tikzpicture}[scale=1]
\draw [thin] (-2,0)--(4,0);
\draw [thin] (0,-0.5)--(0,3.2);

\draw [ultra thick, domain=-2:0]   plot(\x, {3.2});

\draw [ultra thick, domain=0.001:4]     plot(\x, {1.2 - \x + \x*ln(\x/1.2)});
\draw [ultra thick, domain=0.001:4]     plot(\x, {2.2 - \x + \x*ln(\x/2.2)});

\draw [thin] (1.2,0.1)--(1.2,-0.1);
\draw [thin] (2.2,0.1)--(2.2,-0.1);

\node [above] at (1.2,-0.6) {$\gamma_{1}$};
\node [above] at (2.2,-0.6) {$\gamma_{2}$};

\node [right] at (4,2.016)  {$\ell \left( \gamma_{1} ; \cdot \right)$};
\node [right] at (4,0.591)  {$\ell \left( \gamma_{2} ; \cdot \right)$};

\node [left] at (-2,3.2) {$\infty$};
\node [below left] at (0,0) {$0$};

\end{tikzpicture}
\caption{Graphs of the functions $\ell \left( \gamma_{1} ; \cdot \right)$ and $\ell \left( \gamma_{2} ; \cdot \right)$ for $ 0 < \gamma_{1} < \gamma_{2} < \infty$}
\label{fig:graphell}
\end{figure}

Because in the present case $I$ is just an infimum of Poisson rate functions, these inequalities imply that $I$ has some special properties. They are described in the following proposition.
\begin{proposition}
\label{pr:propertiesofI}
In the present case, $I \left( a \right) = 0$ if and only if $a \in \mathcal{R} \left( t \right)$. If $I \left( a \right) > 0$ for some $a \in \mathbb{R}$, then exactly one of the following three scenerios is true:
\begin{enumerate}
\item $a < c_{-} = \inf \mathcal{R} \left( t \right)$ and $I \left( b \right) = \ell \left( c_{-} ; b \right)$ for all $b \in \left( -\infty,c_{-} \right]$;
\item $a > c_{+} = \sup \mathcal{R} \left( t \right)$ and $I \left( b \right) = \ell \left( c_{+} ; b \right)$ for all $b \in \left[ c_{+},\infty \right)$;
\item the previous two cases do not hold and $I \left( b \right) = \min \left\lbrace \ell \left( c_{-} ; b \right) , \ell \left( c_{+} ; b \right) \right\rbrace$ for all $b \in \left[ c_{-} , c_{+} \right]$, where $c_{-} = \sup \left( \mathcal{R} \left( t \right) \cap \left( -\infty,a \right) \right)$ and $c_{+} = \sup  \left( \mathcal{R} \left( t \right) \cap \left( a,\infty \right) \right)$.
\end{enumerate}
\end{proposition}
\begin{proof}
It follows immediately from equations \eqref{eq:Iunscaled} and \eqref{eq:poisratefunc} that $I \left( a \right) = 0$ if and only if $a \in \mathcal{R} \left( t \right)$. Hence, $I \left( a \right) > 0$ implies that the distance of $a$ to $\mathcal{R} \left( t \right)$ is strictly positive, since $\mathcal{R} \left( t \right)$ is closed. The three scenarios now follow from the inequalities \eqref{eq:ell_leq} and  \eqref{eq:ell_geq}.
\end{proof}
The previous proposition may seem rather abstract. To get some intuition, the following example describes a typical rate function.
\begin{example}
\label{ex:visualratefun}
Suppose that $\mathcal{R} \left( t \right) = \left[ \alpha,\beta \right] \cup \left[ \gamma,\delta \right]$ for some $0 < \alpha < \beta < \gamma < \delta < \infty$. Then the function $I$ looks like the graph shown in Figure~\ref{fig:visualI}: it equals $0$ on the intervals $[\alpha,\beta]$ and $[\gamma,\delta]$, whereas it equals the minimum of $\ell \left( \beta ; \cdot \right)$ and $\ell \left( \gamma ; \cdot \right)$ on the interval $\left( \beta,\gamma \right)$ in between. On the interval $\left( -\infty,\alpha \right]$ the function $I$ equals $\ell \left( \alpha ; \cdot \right)$ and on the interval $\left[ \delta,\infty \right)$ the function $I$ equals $\ell \left( \delta ; \cdot \right)$.
\end{example}

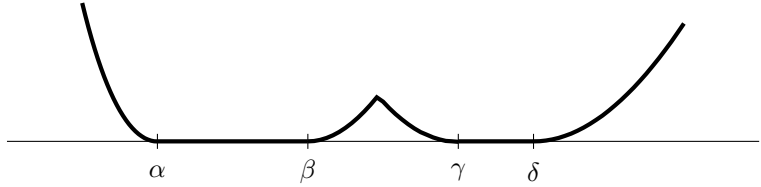
\begin{figure}
\centering
\begin{tikzpicture}[scale=2]
\draw [thin] (0,0)--(5,0);
\draw [thin] (1,0.05)--(1,-0.05);
\draw [thin] (2,0.05)--(2,-0.05);
\draw [thin] (3,0.05)--(3,-0.05);
\draw [thin] (3.5,0.05)--(3.5,-0.05);

\node at (1,-0.2)   {$\alpha$};
\node at (2,-0.2)   {$\beta$};
\node at (3,-0.2)   {$\gamma$};
\node at (3.5,-0.2) {$\delta$};

\draw [ultra thick, domain=0.5:1]   plot(\x, {6*( 1 - \x + \x*ln(\x/1) )});
\draw [ultra thick, domain=1:2]     plot(\x, {0});
\draw [ultra thick, domain=2:3]     plot(\x, {6*min( 2 - \x + \x*ln(\x/2) , 3 - \x + \x*ln(\x/3)});
\draw [ultra thick, domain=3:3.5]   plot(\x, {0});
\draw [ultra thick, domain=3.5:4.5] plot(\x, {6*( 3.5 - \x + \x*ln(\x/3.5) )});
\end{tikzpicture}
\caption{Visualization of the function $I$ in Example~\ref{ex:visualratefun}}
\label{fig:visualI}
\end{figure}

To compute $\mathcal{R} \left( t \right)$, it is often convenient to use the following properties of step functions in $D \left( \left[ 0,\infty \right) ; \mathcal{E} \right)$. (For the definition of a step function, see Section~\ref{sec:stepskor}.) Recall that the set of all step functions in $D \left( \left[ 0,\infty \right) ; \mathcal{E} \right)$ is denoted by $\mathcal{S} \left( \left[ 0,\infty \right) ; \mathcal{E} \right)$. 
\begin{lemma}
\label{lem:rtclofstep}
If $\left\lbrace \phi_{t} \left( f \right) \middle\vert f \in \mathcal{S} \left( \left[ 0,\infty \right) ; \mathcal{E} \right) \right\rbrace \subset \mathcal{R} \left( t \right)$, then
\begin{align*}
\mathcal{R} \left( t \right) = \mathrm{cl} \left\lbrace \phi_{t} \left( f \right) \middle\vert f \in \mathcal{S} \left( \left[ 0,\infty \right) ; \mathcal{E} \right) \right\rbrace = \left\lbrace \phi_{t} \left( f \right) \middle\vert f \in \mathcal{D} \left( \left[ 0,\infty \right) ; \mathcal{E} \right) \right\rbrace.
\end{align*}
\end{lemma}
\begin{proof}
This follows from Lemma~\ref{lem:Rnonempty}, Corollary~\ref{co:stepdense} and Lemma~\ref{lem:phictu}.
\end{proof}
\begin{lemma}
\label{lem:rtclosedint}
If $\left\lbrace \phi_{t} \left( f \right) \middle\vert f \in \mathcal{S} \left( \left[ 0,\infty \right) ; \mathcal{E} \right) \right\rbrace \subset \mathcal{R} \left( t \right)$, then $\mathcal{R} \left( t \right)$ is a closed interval.
\end{lemma}
\begin{proof}
It suffices to show that $\mathcal{R} \left( t \right)$ is convex. Let $f^{1}_{\mathrm{c}}, f^{2}_{\mathrm{c}} \in  \mathcal{S} \left( \left[ 0,\infty \right) ; \mathcal{E} \right)$. We may assume that $\phi_{t} \left( f^{1}_{\mathrm{c}} \right) \leq \phi_{t} \left( f^{2}_{\mathrm{c}} \right)$. For $x \in \left[ 0,t \right]$ we define the function $g_{x}$ via
\begin{align*}
g_{x} \left( s \right) = \ind{s < x} f^{1}_{\mathrm{c}} \left( s \right) + \ind{s \geq x} f^{2}_{\mathrm{c}} \left( s \right)
\end{align*}
for $s \in \left[ 0,\infty \right)$. Clearly, $g_{x} \in \mathcal{S} \left( \left[ 0,\infty \right) ; \mathcal{E} \right)$ and
\begin{align*}
\phi_{t} \left( g_{x} \right) = \phi_{x} \left( f^{1}_{\mathrm{c}} \right) + \left( \phi_{t} \left( f^{2}_{\mathrm{c}} \right) - \phi_{x} \left( f^{2}_{\mathrm{c}} \right) \right).
\end{align*}
Using the continuity of the integral and applying the Intermediate Value Theorem, it follows that
\begin{align*}
\left[ \phi_{t} \left( f^{1}_{\mathrm{c}} \right) , \phi_{t} \left( f^{2}_{\mathrm{c}} \right) \right] \subset \left\lbrace \phi_{t} \left( g_{x} \right) \middle\vert x \in \left[ 0,t \right] \right\rbrace \subset \mathcal{R} \left( t \right).
\end{align*}
Combined with Lemma~\ref{lem:rtclofstep}, this implies the statement of the lemma.
\end{proof}

Let $f_{\mathrm{c}} \in \mathcal{S} \left( \left[ 0,\infty \right) ; \mathcal{E} \right)$ be a step function. Clearly, $f_{\mathrm{c}}$ has a unique minimal representation $\left\lbrace \left( t_{i} , \alpha_{i} \right) \right\rbrace_{i=0}^{k}$, where $k \in \mathbb{N}$, $0 = t_{0} < t_{1} < \ldots < t_{k} < \infty$ and $\alpha_{0} , \ldots , \alpha_{k} \in \mathcal{E}$ are such that $f_{\mathrm{c}} \left( t \right) = \alpha_{i}$ for $t \in \left[ t_{i} , t_{i+1} \right)$ and $i = 0 , \ldots , k-1$ and $f_{\mathrm{c}} \left( t \right) = \alpha_{k}$ for $t \in \left[ t_{k} , \infty \right)$. Given this minimal representation, we define its truncated minimal step size by
\begin{align*}
\Delta_{f_{\mathrm{c}}} = 1 \wedge \min_{i = 1 , \ldots , k} \left\lbrace t_{i} - t_{i-1} \right\rbrace.
\end{align*}
Additionally, we define $t_{k+1} = t_{k} \vee t$. The truncated minimal step size and $t_{k+1}$ will be used for computing attainable parameters.

In the upcoming examples, we would like to compute rate functions via attainable parameters. To compute attainable parameters, we use the following strategy. We fix a certain path $f$, often a step function. This gives us a parameter value $\phi_{t} \left( f \right)$. Then we would like to show that, with positive probability, the background process stays `close' to $f$, which will imply that $\phi_{t} \left( f \right)$ is an attainable parameter. 

Staying `close' to $f$ depends on properties of $\mathcal{E}$ and the background process. In most cases, the background process needs a little bit of room (both in time and in space) to jump near a discontinuity of $f$. This is where the truncated minimal step size comes in: it is an upper bound on the time we give the background process for jumping near a discontinuity of a step function. The precise meaning of this will become clearer in the examples.

The first example treats the familiar case of a Markov-modulated infinite-server queue, i.e., the case in which the background process is an irreducible Markov chain. This case is partly studied in \cite{BDKM2014} (Model~I) and \cite{BM2013} (Model~II). In the example, we recover \cite[Th.~2]{BDKM2014} and \cite[Th.~1]{BM2013}. Additionally, we generalize these results to our model and extend them to a full LDP.
\begin{example}
\label{ex:markovmod}
Let $J$ be an irreducible, continuous-time Markov process with finite state space $\mathcal{E} = \left\lbrace 1 , \ldots , d \right\rbrace$ and consider the modulated infinite-server queue $\left( J , \lambda , \kappa , \mu \right)$. Given the scaling $\lambda \mapsto n \lambda$, Theorem~\ref{thm:LDP} (combined with Lemma~\ref{lem:unscaledldp}) shows that $\frac{1}{n} M_{n} \left( t \right)$ satisfies an LDP with rate function $I$. This rate function may be computed as follows.

Note that $D \left( \left[ 0,\infty \right) ; \mathcal{E} \right) = \mathcal{S} \left( \left[ 0,\infty \right) ; \mathcal{E} \right)$, since $\mathcal{E}$ is finite. Fix any function $g \in D \left( \left[ 0,\infty \right) ; \mathcal{E} \right)$ with minimal representation $\left\lbrace \left( t_{i} , \alpha_{i} \right) \right\rbrace_{i=0}^{k}$ and take any $\epsilon \in \left( 0 , 1 \right)$. 

Define $\mathcal{W} \left( g ; \epsilon \right)$ as the set of all $f \in D \left( \left[ 0,\infty \right) ; \mathcal{E} \right)$ such that
\begin{align*}
f \left( t \right) &= \alpha_{i-1} & &\, \forall  t \in \left[ t_{i-1} + \tfrac{\epsilon}{2} \tfrac{1}{k} \Delta_{g} , t_{i} - \tfrac{\epsilon}{2} \tfrac{1}{k} \Delta_{g} \right) \quad \forall i \in \left\lbrace 1 , \ldots , k \right\rbrace,\\
f \left( t \right) &= \alpha_{k} & &\, \forall  t \in \left[ t_{k} , t_{k+1} \right].
\end{align*}
Now note that
\begin{align*}
\sup_{f \in \mathcal{W} \left( g ; \epsilon \right)} \phi_{t} \left( f \right) \leq \phi_{t} \left( g \right) + \epsilon \max_{j \in \left\lbrace 1 , \ldots , d \right\rbrace} \lambda \left( j \right)
\intertext{and}
\inf_{f \in \mathcal{W} \left( g ; \epsilon \right)} \phi_{t} \left( f \right) \geq \phi_{t} \left( g \right) - \epsilon \max_{j \in \left\lbrace 1 , \ldots , d \right\rbrace} \lambda \left( j \right),
\end{align*}
so we can get both the supremum and the infimum arbitrarily close to $\phi_{t} \left( g \right)$ by taking $\epsilon$ small enough.

Observe that $\mathbb{P} \left( J \in \mathcal{W} \left( g ; \epsilon \right) \right) > 0$, thanks to the irreducibility of $J$. Consequently, $\mathcal{R} \left( t \right) = \left\lbrace \phi_{t} \left( g \right) \, \middle\vert \, g \in D \left( \left[ 0,\infty \right) ; \mathcal{E} \right) \right\rbrace$. Then Lemma~\ref{lem:rtclosedint} implies that $\mathcal{R} \left( t \right)$ is a closed interval. Using that $\mathcal{E}$ is finite, we immediately get
\begin{align*}
\mathcal{R} \left( t \right) = \left[ a_{-} , a_{+} \right],
\end{align*}
where $0 \leq a_{-} \leq a_{+} < \infty$ with $a_{-} = \inf_{g \in D \left( \left[ 0,\infty \right) ; \mathcal{E} \right)} \phi_{t} \left( g \right)$ and $a_{+} = \sup_{g \in D \left( \left[ 0,\infty \right) ; \mathcal{E} \right)} \phi_{t} \left( g \right)$. Now applying Proposition~\ref{pr:propertiesofI}, it follows that the rate function $I$ is given by
\begin{align}
I \left( a \right) = \begin{cases} \infty & a \in \left( -\infty,0 \right);\\ \ell \left( a_{-} ; a \right) & a \in \left[ 0,a_{-} \right];\\ 0 & a \in \left[ a_{-} , a_{+} \right];\\ \ell \left( a_{+} ; a \right) & a \in \left[ a_{+},\infty \right). \end{cases}
\end{align}
\end{example}
The result of the previous example depends neither on the initial distribution nor on the transition rate matrix of the irreducible Markov chain. Moreover, the analysis in the previous example implies the following lemma. It shows that we always obtain a good rate function when the background process has a finite state space.
\begin{lemma}
\label{lem:finitestatespace}
Let $J^{\left( 1 \right)}$ be a background process with finite state space $\mathcal{E}$ and let $J^{\left( 2 \right)}$ be an irreducible Markov chain with the same state space. Consider the two modulated infinite-server queues $\left( J^{\left( 1 \right)} , \lambda , \kappa , \mu \right)$ and $\left( J^{\left( 2 \right)} , \lambda , \kappa , \mu \right)$. Scaling $\lambda \mapsto n \lambda$, we obtain in both cases an LDP for the number of jobs in the system with corresponding rate functions $I^{\left( 1 \right)}$ and $I^{\left( 2 \right)}$. Then it holds that
$I^{\left( 1 \right)} \left( a \right) \geq I^{\left( 2 \right)} \left( a \right)$ for all $a \in \mathbb{R}$. In particular, both $I^{\left( 1 \right)}$ and $I^{\left( 2 \right)}$ are good rate functions.
\end{lemma}
In the next example we will modulate an infinite-server queue by another Markov-modulated infinite-server queue. This setup differs from the setup considered in \cite{BDKM2014} and \cite{BM2013}. In particular, the state space of the background process is countably infinite, so that we may obtain a rate function that is not good.
\begin{example}
Consider a Markov-modulated infinite-server queue as described in \cite{OP1986}, i.e., a Markov-modulated infinite-server queue under the assumptions of Model~I. Assume that neither the arrival rates nor the server work rates are identically equal to $0$ and that the system starts empty. Let $J \left( t \right)$ be the number of jobs in this Markov-modulated infinite-server queue at time $t \geq 0$. Then $J$ is a c\`{a}dl\`{a}g stochastic process and its state space is $\mathcal{E} = \mathbb{Z}_{\geqslant 0}$.

Consider the modulated infinite-server queue $\left( J , \lambda , \kappa , \mu \right)$ and impose the scaling $\lambda \mapsto n \lambda$. Then $\frac{1}{n} M_{n} \left( t \right)$ satisfies an LDP with rate function $I$, according to Theorem~\ref{thm:LDP} and Lemma~\ref{lem:unscaledldp}. This rate function may be computed as follows.

Recall that $J$ stays in state $m \in \mathcal{E}$ during $\left[ t , t + \Delta t \right]$ with positive probability for arbitrarily large $\Delta t$. Moreover, because neither the arrival rates nor the server work rates are identically equal to $0$, the process $J$ also has the following property. If $J \left( t \right) = m_{1}$ at time $t \geq 0$, then it jumps to state $m_{2} \in \mathcal{E}$ during $\left[ t , t + \Delta t \right]$ with positive probability for arbitrarily small $\Delta t$.

Roughly speaking, these two properties mean that the background process is irreducible, in the sense that it can jump to or stay in any state during any time interval we would like. Of course, this is very similar to the Markov chain being irreducible in the previous example. Consequently, our strategy for determining the attainable parameters will be very similar, although there are some subtleties related to the state space being infinite.

Fix any $g \in \mathcal{S}  \left( \left[ 0,\infty \right) ; \mathcal{E} \right)$ with minimal representation $\left\lbrace \left( t_{i} , \alpha_{i} \right) \right\rbrace_{i=0}^{k}$ and take any $\epsilon \in \left( 0,1 \right)$. Let $\mathcal{W} \left( g;\epsilon \right)$ denote the set of all $f \in D  \left( \left[ 0,\infty \right) ; \mathcal{E} \right)$ with
\begin{align*}
f \left( t \right) &= \alpha_{i-1} & &\, \forall  t \in \left[ t_{i-1} +  \tfrac{\epsilon}{2} \tfrac{1}{k} \Delta_{g} , t_{i} - \tfrac{\epsilon}{2} \tfrac{1}{k} \Delta_{g} \right) \quad \forall i \in \left\lbrace 1 , \ldots , k \right\rbrace,\\
f \left( t \right) &= \alpha_{k} & &\, \forall  t \in \left[ t_{k} , t_{k+1} \right],
\end{align*}
and
\begin{align*}
0 \leq f \left( t \right) \leq \alpha_{0} & & &\, \forall t \in \left[ 0,\tfrac{\epsilon}{2} \tfrac{1}{k} \Delta_{g} \right),\\
\alpha_{i-1} \wedge \alpha_{i} \leq f \left( t \right) \leq \alpha_{i-1} \vee \alpha_{i} & & &\, \forall t \in \left[ t_{i} - \tfrac{\epsilon}{2} \tfrac{1}{k} \Delta_{g} , t_{i} + \tfrac{\epsilon}{2} \tfrac{1}{k} \Delta_{g} \right)\\
& & & \, \forall i \in \left\lbrace 0 , \ldots , k-1 \right\rbrace,\\
\alpha_{k-1} \wedge \alpha_{k} \leq f \left( t \right) \leq \alpha_{k-1} \vee \alpha_{k} & & &\, \forall t \in \left[ t_{k} - \tfrac{\epsilon}{2} \tfrac{1}{k} \Delta_{g} , t_{k} \right].
\end{align*}
Then we have
\begin{align*}
\sup_{f \in \mathcal{W} \left( g ; \epsilon \right)} \phi_{t} \left( f \right) \leq \phi_{t} \left( g \right) + \epsilon \max_{i \in \left\lbrace 0 , \ldots , k \right\rbrace} \max_{j \in \left\lbrace 0 , \ldots , \alpha_{i} \right\rbrace} \lambda \left( j \right)
\intertext{and}
\inf_{f \in \mathcal{W} \left( g ; \epsilon \right)} \phi_{t} \left( f \right) \geq \phi_{t} \left( g \right) - \epsilon \max_{i \in \left\lbrace 0 , \ldots , k \right\rbrace} \max_{j \in \left\lbrace 0 , \ldots , \alpha_{i} \right\rbrace} \lambda \left( j \right).
\end{align*}

The two properties of the background process described above imply that $\mathbb{P} \left( J \in \mathcal{W} \left( g ; \epsilon \right) \right) > 0$. It follows that $\left\lbrace \phi_{t} \left( g \right) \, \middle\vert \, g \in \mathcal{S} \left( \left[ 0,\infty \right) ; \mathcal{E} \right) \right\rbrace \subset \mathcal{R} \left( t \right)$.
Write $a_{-} = \inf_{g \in D \left( \left[ 0,\infty \right) ; \mathcal{E} \right)} \phi_{t} \left( g \right)$ and $a_{+} = \sup_{g \in D \left( \left[ 0,\infty \right) ; \mathcal{E} \right)} \phi_{t} \left( g \right)$. Lemma~\ref{lem:rtclofstep} and Lemma~\ref{lem:rtclosedint} imply that $\mathcal{R} \left( t \right) = \left[ a_{-} , a_{+} \right]$ if $a_{+} < \infty$ and $\mathcal{R} \left( t \right) = \left[ a_{-} , \infty \right)$ if $a_{+} = \infty$. Hence,
\begin{align}
I \left( a \right) = \begin{cases} \infty & a \in \left( -\infty,0 \right);\\ \ell \left( a_{-} ; a \right) & a \in \left[ 0,a_{-} \right];\\ 0 & a \in \left[ a_{-} , a_{+} \right];\\ \ell \left( a_{+} ; a \right) & a \in \left[ a_{+},\infty \right) \end{cases}
\end{align}
if $a_{+} < \infty$ and
\begin{align}
I \left( a \right) = \begin{cases} \infty & a \in \left( -\infty,0 \right);\\ \ell \left( a_{-} ; a \right) & a \in \left[ 0,a_{-} \right];\\ 0 & a \in \left[ a_{-} , \infty \right) \end{cases}
\end{align}
if $a_{+} = \infty$. Note that $I$ is not a good rate funtion if $a_{+} = \infty$.
\end{example}
The previous example only depends on the state space being countable and discrete and on the background process being irreducible in the sense described above. Consequently, the same result holds for irreducible Markov processes with a countable, discrete state space.

In the last example of this section we compare rate functions that are obtained using two different background processes. One background process is a Markov chain, whereas the other background process is a reflected Brownian motion, which has an uncountable state space. It turns out that both background processes lead to the same LDP, even though the background processes are completely different. Apparently, two very different modulating processes may lead to the same rate function for the LDP, even if the arrival rates, service requirements and server work rates are nontrivial.
\begin{example}
Let $\mathcal{E} = \left[ 0,1 \right]$ be equipped with the Euclidean metric. Define $\lambda \colon \left[ 0,1 \right] \to \left[ 0,1 \right]$ by $\lambda \left( x \right) = x$, $\kappa \colon \left[ 0,1 \right] \to \left[ 0,1 \right]$ by $\kappa \left( x \right) = 1$ and $\mu \colon \left[ 0,1 \right] \to \left[ 0,1 \right]$ by $\mu \left( x \right) = 1 - x$.

Let $J^{\mathrm{MC}}$ be an irreducible, continuous-time Markov chain with state space $\left\lbrace 0,1 \right\rbrace$.
Let $J^{\mathrm{rBM}}$ be a reflected Brownian motion with reflecting barriers $0$ and $1$. For simplicity, assume that $J^{\mathrm{rBM}}$ starts in $x_{0} \in \left( 0,1 \right)$, so
\begin{align*}
J^{\mathrm{rBM}} \left( t \right) = x_{0} + W \left( t \right) + L \left( t \right) - U \left ( t \right)
\end{align*}
for some standard Brownian motion $W$, lower-regulator process $L$ and upper-regulator process $U$ (see for instance \cite{DIKM2012}).

Consider the two modulated infinite-server queues $\left( J^{\mathrm{MC}} , \lambda , \kappa , \mu \right)$ and $\left( J^{\mathrm{rBM}} , \lambda , \kappa , \mu \right)$. Under the scaling $\lambda \mapsto n \lambda$, both $\frac{1}{n} M_{n}^{\mathrm{rBM}} \left( t \right)$ and $\frac{1}{n} M_{n}^{\mathrm{MC}} \left( t \right)$ satisfy an LDP with the same good rate function $I$, which is given by
\begin{align}
I \left( a \right) = \begin{cases} \infty & a \in \left( -\infty,0 \right);\\ 0 & a \in \left[ 0,t \right];\\ \ell \left( t ; a \right) & a \in \left[ t,\infty \right).\end{cases}
\end{align}

The rate function for the LDP corresponding to $\frac{1}{n} M_{n}^{\mathrm{MC}} \left( t \right)$ is derived in Example~\ref{ex:markovmod}. It is easy to see that the rate function has the form claimed above.

We will show that $\frac{1}{n} M_{n}^{\mathrm{rBM}} \left( t \right)$ satisfies an LDP with the same rate function. Fix $g \in \mathcal{S} \left( \left[ 0,\infty \right) ; \mathcal{E} \right)$ with minimal representation $\left\lbrace \left( t_{i} , \alpha_{i} \right) \right\rbrace_{i=0}^{k}$ and take any $\epsilon > 0$. Define $\mathcal{W} \left( g ; \epsilon \right)$ as the set of all $f \in D  \left( \left[ 0,\infty \right) ; \mathcal{E} \right)$ such that
\begin{align*}
| f \left( t \right) - \alpha_{i} | \leq \epsilon \quad \forall  t \in \left[ t_{i-1} + \tfrac{\epsilon}{2} \tfrac{1}{k} \Delta_{g} , t_{i} - \tfrac{\epsilon}{2} \tfrac{1}{k} \Delta_{g} \right) \quad \forall i \in \left\lbrace 1 , \ldots , k \right\rbrace.
\end{align*}
Then we get
\begin{align*}
\sup_{f \in \mathcal{W} \left( g ; \epsilon \right)} \phi_{t} \left( f \right) \leq \phi_{t} \left( g \right) + \epsilon t  + \epsilon
\intertext{and}
\inf_{f \in \mathcal{W} \left( g ; \epsilon \right)} \phi_{t} \left( f \right) \geq \phi_{t} \left( g \right) - \epsilon t  - \epsilon.
\end{align*}
Now observe that
\begin{align*}
\mathbb{P} \left( J^{\mathrm{rBM}} \in \mathcal{W} \left( g ; \epsilon \right) \right) \geq
\mathbb{P} \left( x_{0} + W \in \mathcal{W} \left( g ; \epsilon \right) \right) > 0,
\end{align*}
due to the definition of $J^{\mathrm{rBM}}$ and $W$ being a Brownian motion.

It follows that $\left\lbrace \phi_{t} \left( g \right) \, \middle\vert \, g \in \mathcal{S} \left( \left[ 0,\infty \right) ; \mathcal{E} \right) \right\rbrace \subset \mathcal{R}^{\mathrm{rBM}} \left( t \right)$, so $\mathcal{R}^{\mathrm{rBM}} \left( t \right) = \left[ 0,t \right]$ and the corresponding rate function is given by the function $I$ above.
\end{example}

In this section we considered examples in which the background process was not scaled. As shown, this implies some special properties, which we can use to explicitly compute rate functions. In the next section, we will scale the background process, too. Although explicit computations are not possible in general, there are still cases for which we may derive rate functions.

\section{Examples: scaled background processes}
In this section we will give two examples in which the background process is scaled. In the first example, we will consider the Markov-modulated infinite-server queue and derive an explicit rate function under a superlinear time-scaling. In the second example, we will consider a new model in which the background process is a Brownian motion. In this case, the rate function will be given as the solution of a variational problem.

\begin{example}
Let $J$ be an irreducible continuous-time Markov chain with finite state space $\left\lbrace 1 , \dotsc , d \right\rbrace$ and generator matrix $Q$. Denote the corresponding stationary distribution by $\pi = \left( \pi_{1} , \dotsc , \pi_{d} \right)$. 

Consider the modulated infinite-server queue $\left( J , \lambda , \kappa , \mu \right)$. Define $\mu_{\infty} = \sum_{j=1}^{d} \pi_{j} \mu_{j}$ and 
\begin{align*}
\varrho_{t} = \sum_{j = 1}^{d} \pi_{j} \lambda_{j} \int_{0}^{t} e^{- \kappa_{j} \mu_{\infty} \left( t-s \right)} \, \dd s 
= \sum_{j = 1}^{d} \pi_{j} \frac{\lambda_{j}}{\kappa_{j} \mu_{\infty}} \left( 1 - e^{- \kappa_{j} \mu_{\infty} t } \right).
\end{align*}
Scale $\lambda \mapsto n \lambda$ and $J \mapsto J_{n}$, where $J_{n} \left( t \right) = J \left( n^{1+\epsilon} t \right)$. It is easy to see that scaling $J \mapsto J_{n}$ is equivalent to scaling $Q \mapsto n^{1+\epsilon} Q$.

The sequence of random parameters $\left\lbrace \phi_{t} \left( J_{n} \right) \right\rbrace_{n \in \mathbb{N}}$ satisfies an LDP with rate function $\psi$, where
\begin{align*}
\psi \left( a \right) = \begin{cases} 0 & a = \varrho_{t}; \\ \infty & a \not= \varrho_{t}. \end{cases}
\end{align*}
Indeed, this follows from the fact that
\begin{align*}
\liminf_{n\to \infty} \frac{1}{n} \log \mathbb{P} \left( \phi_{t} \left( J_{n} \right) \in \openball{\rho_{t}}{\eta} \right) = 0
\end{align*}
and
\begin{align*}
\limsup_{n\to \infty} \frac{1}{n} \log \mathbb{P} \left( \phi_{t} \left( J_{n} \right) \not\in \openball{\rho_{t}}{\eta} \right) = - \infty
\end{align*}
for all $\eta > 0$. These equalities are an immediate result from the proof of \cite[Th.~3]{BDKM2014}.

Given this LDP for $\left\lbrace \phi_{t} \left( J_{n} \right) \right\rbrace_{n \in \mathbb{N}}$, Theorem~\ref{thm:LDP} implies that $\frac{1}{n} M_{n} \left( t \right)$ satisfies an LDP with rate function $I$, where
\begin{align*}
I \left( a \right) = \ell \left( \varrho_{t} ; a \right).
\end{align*}
Hence, under this superlinear time-scaling of the background Markov chain, the LDP for $\frac{1}{n} M_{n} \left( t \right)$ is governed by a Poisson rate function with parameter $\varrho_{t}$.
\end{example} 
\begin{example}
Consider a modulated infinite-server queue $\left( J , \lambda , \kappa , \mu \right)$, where the background process $J$ is a standard Brownian motion $W$ on $\left[ 0,\infty \right)$. By $\overline{W}$ we denote its restriction to the interval $\left[ 0,t \right]$. The sample paths of $\overline{W}$ are elements of $C_{0} \left[ 0,t \right]$, the space of continuous functions $f \colon \left[ 0,t \right] \to \mathbb{R}$ with $f \left( 0 \right) = 0$.

Equip $C_{0} \left[ 0,t \right]$ with the supremum metric. Of course, we may view the function $\phi_{t}$ as a map from $C_{0} \left[ 0,t \right]$ to $\left[ 0,\infty \right)$ and this map is continuous under the supremum metric.

Scale $\lambda \mapsto n \lambda$ and $J \mapsto J_{n}$, where $J_{n}$ is given by a linear time-scaling: $J_{n} \left( s \right) = W \left( s/n \right)$ for $s \geq 0$. Under this scaling, the arrivals are sped up linearly, whereas the time scale of the Brownian motion is slowed down linearly.

This scaling resembles the scaling featured in \cite{DM2014}. There, the authors considered a modulated infinite-server queue under a linear scaling of both the arrival rate and the time scale of an irreducible Markov chain. The rate function obtained in \cite{DM2014} is given as the solution of a variational problem. We will obtain a similar result in this example.

Since $W$ is a Brownian motion, we have
\begin{align*}
\phi_{t} \left( J_{n} \right) \overset{\mathrm{d}}{=} \phi_{t} \left( \frac{1}{\sqrt{n}} W \right) = \phi_{t} \left( \frac{1}{\sqrt{n}} \overline{W} \right).
\end{align*}
Schilder's Theorem (cf.\ \cite[Th.~5.2.3]{dz1998}) states that $\frac{1}{\sqrt{n}} \overline{W}$ satisfies an LDP in $C_{0} \left[ 0,t \right]$ with good rate function
\begin{align*}
\xi \left( f \right) = \begin{cases} \frac{1}{2} \int_{0}^{t} \vert \dot{f} \left( s \right) \vert^{2} \, \dd s & f \in H_{1} \left( \left[ 0,t \right] \right); \\ \infty & \text{else}.\end{cases}
\end{align*}
Here, $H_{1} \left( \left[ 0,t \right] \right)$ denotes the set of all absolutely continuous functions $f \in C_{0} \left[ 0,t \right]$ that have square integrable derivative $\dot{f}$.

The contraction principle (cf.\ \cite[Th.~4.2.1]{dz1998}) now implies that $\phi_{t} \left( J_{n} \right)$ satisfies an LDP with good rate function $\psi$, where $\psi$ is given by
\begin{align*}
\psi \left( a \right) = \inf \left\lbrace \xi \left( f \right) \middle\vert f \in H_{1} \left( \left[ 0,t \right] \right), \phi_{t} \left( f \right) = a \right\rbrace.
\end{align*}
It follows from Theorem~\ref{thm:LDP} that $\frac{1}{n} M_{n} \left( t \right)$ satisfies an LDP with rate function $I$, where $I$ is given by
\begin{align*}
I \left( a \right) = \inf_{\gamma \in \mathcal{R} \left( t \right)} \left[ \ell \left( \gamma ; a \right) + \psi \left( \gamma \right) \right].
\end{align*}
Now recall that $\left\lbrace \psi < \infty \right\rbrace \subset \mathcal{R} \left( t \right)$. Also observe that $\left\lbrace \xi < \infty \right\rbrace = H_{1} \left( \left[ 0,t \right] \right)$ and that $\left\lbrace \psi < \infty \right\rbrace = \left\lbrace \phi_{t} \left( f \right) \middle\vert f \in H_{1} \left( \left[ 0,t \right] \right) \right\rbrace$. Then we may rewrite $I$ as
\begin{align*}
I \left( a \right) 
&= \inf_{\gamma \in \left\lbrace \psi < \infty \right\rbrace} \left[ \ell \left( \gamma ; a \right) + \psi \left( \gamma \right) \right]\\
&= \inf_{f \in H_{1} \left( \left[ 0,t \right] \right)} \left[ \ell \left( \phi_{t} \left( f \right) ; a \right) + \psi \left( \phi_{t} \left( f \right) \right) \right].
\end{align*}
Hence, $I$ is given as the solution of a variational problem.
\end{example}

\section{Discussion and concluding remarks}
\label{sec:discussion}
In this paper, we studied an infinite-server queue in a random environment and proved a full LDP for the transient number of jobs in the system. The proof of this LDP has two essential ingredients, namely the result that the transient number of jobs in the system has a Poisson distribution with a random parameter and the assumption that the random parameter satisfies an LDP. Hence, the large deviations behavior of the random parameter seems to be the crucial factor that determines the large deviations behavior of the number of jobs in the system.

The rate function corresponding to the LDP for the number of jobs is rather abstract. Nevertheless, we showed in the examples how to compute the rate function in certain specific cases. In particular, we recovered earlier obtained results for Markov-modulated infinite-server queues and strengthened these to a full LDP. Additionally, we proved LDPs when the background process has an uncountable state space. In all examples, knowledge about the behavior of the background process could be exploited to describe the rate function.

There are several interesting topics for future research on the modulated infinite-server queue presented here. In this paper, we only looked at large deviations of the number of jobs at a fixed time $t \geq 0$. However, for certain applications it may be desirable to know the deviations over the whole time interval $\left[ 0,t \right]$. Therefore, it would be interesting to consider sample path large deviations.
Also moderate deviations could be worth investigating, so as to bridge the gap between the central limit theorems and the large deviations results for modulated infinite-server queues.

Furthermore, it could be very interesting to study other models. In particular, more general arrival processes or service time distributions could be considered. As an example, it seems that the setup of \cite{BM2013} could be generalized to include a background process with countable state space and general service time distributions. Regarding general service time distributions, it should be mentioned that there might be some measurability issues when the background process has an uncountable state space. Finally, it would be interesting to see whether the large deviations results for modulated infinite-server queues carry over to modulated Ornstein-Uhlenbeck processes. To the best of our knowledge, this has not been investigated so far.

\begin{acknowledgement}
This research has been partly funded by the Interuniversity Attraction Poles Programme initiated by the Belgian Science Policy Office.
\end{acknowledgement}

\appendix
\section{Transient number of jobs in the system}
\label{sec:transient}
In this section, we provide the precise mathematical description of the model and determine the distribution of the number of jobs in the system at time $t \geq 0$, which is denoted by $M \left( t \right)$. We mentioned in Section~\ref{sec:intro} that the steady-state distribution of the number of jobs in the system has already been determined for specific background processes in Model~I and Model~II. However, in this case we would like to determine the transient distribution given a general background process for the model described below, which generalizes Model~I and Model~II. Fortunately, the setup of our model is quite convenient and we may obtain the transient distribution without too much effort.

By $D \left( \left[ 0, \infty \right) ; \mathcal{E} \right)$ we denote the space of c\`{a}dl\`{a}g functions from $\left[ 0, \infty \right)$ to $\mathcal{E}$, where $\mathcal{E}$ is a metric space with metric $\rho$. We define, in the usual way, a metric $d^{\circ}$ on $D \left( \left[ 0, \infty \right) ; \mathcal{E} \right)$ that generates the Skorokhod $J_{1}$ topology. (For more details, see Section~\ref{sec:stepskor} and references there.)

Let $\left( \Omega , \mathcal{F} , \mathbb{P} \right)$ be a probability space on which we have defined an independent, standard Poisson processes $\overline{Y}$ and an independent, c\`{a}dl\`{a}g stochastic process $J$ with state space $\mathcal{E}$.
Assume that we have defined a collection of independent standard exponential random variables $\overline{Z}_{1}, \overline{Z}_{2}, \ldots$ on this probability space.

To modulate the infinite-server queue, we take continuous functions $\lambda \colon \mathcal{E} \to \left[ 0,\infty \right)$, $\mu \colon \mathcal{E} \to \left[ 0,\infty \right)$ and $\kappa \colon \mathcal{E} \to \left[ 0,\infty \right)$. More precisely, $\lambda \left( J \right)$ modulates the arrival rate, $\kappa \left( J \right)$ modulates the service requirement distribution and $\mu \left( J \right)$ modulates the server work rate.

We define the modulated Poisson process $Y$ via
\begin{align*}
Y \left( t \right) &= \overline{Y} \left( \int_{0}^{t} \lambda \left( J \left( s \right) \right) \, \dd s \right).
\end{align*}
The process $Y$ will be the arrival process. We denote the jump times of $Y$ by $\tau_{1}, \tau_{2}, \ldots$ and the jump times of $\overline{Y}$ by $\overline{\tau}_{1}, \overline{\tau}_{2}, \ldots$. For convenience, we set $\tau_{0} = \overline{\tau}_{0} = 0$. The jump times $\tau_{k}$ and $\overline{\tau}_{k}$ are related via $\tau_{k} = \Lambda^{-} \left( \overline{\tau}_{k} \right)$ and $\overline{\tau}_{k} = \Lambda \left( \tau_{k} \right)$, where $\Lambda \left( t \right) = \int_{0}^{t} \lambda \left( J \left( s \right) \right) \, \dd s$ and $\Lambda^{-} \left( r \right) = \inf \left\lbrace t \geq 0 \, \middle\vert \, \Lambda \left( t \right) \geq r \right\rbrace$.

Define the interarrival times $\sigma_{k} = \tau_{k} - \tau_{k-1}$ and $\overline{\sigma}_{k} = \overline{\tau}_{k} - \overline{\tau}_{k-1}$ for $k \in \mathbb{N}$. For later use, we note that $\overline{\sigma}_{1} , \overline{\sigma}_{2} , \dotsc$ is a sequence of i.i.d.\ random variables with a standard exponential distribution.

At time $t = 0$ there are no jobs in the system. At each jump time of $Y$ exactly one job arrives. Hence, the number of jobs that have entered the system during the time interval $\left[ 0,t \right]$ is given by the (a.s.\ finite) random variable $\sum_{k=1}^{\infty} \ind{\tau_{k} \leq t}$.

When job $k$ enters the system at time $\tau_{k}$, it draws an independent service requirement from an exponential distribution with parameter $\kappa \left( J \left( \tau_{k} \right) \right) \geq 0$, i.e., the service requirement of job $k$ is given by $Z_{k}$, where
\begin{align*}
Z_{k} = \begin{cases} \overline{Z}_{k} / \kappa \left( J \left( \tau_{k} \right) \right) & \text{ if } \kappa \left( J \left( \tau_{k} \right) \right) > 0; \\ \infty & \text{ if } \kappa \left( J \left( \tau_{k} \right) \right) = 0. \end{cases}
\end{align*}
Job $k$ leaves the system when its service requirement has been processed by the server, whose work rate is modulated by the background process $J$ and is equal to $\mu \left( J \left( s \right) \right)$ for $s \geq 0$.

Hence, job $k$ has both entered and left the system before time $t \geq 0$ if and only if $\tau_{k} \leq t$ and $Z_{k} \leq \int_{ \left[\tau_{k} , t \right)} \mu \left( J \left( r \right) \right) \, \dd r$. We get
\begin{align*}
M \left( t \right) = \sum_{k=1}^{\infty} \left( \ind{\tau_{k} \leq t} - \ind{\tau_{k} \leq t} \ind{Z_{k} \leq \int_{ \left[\tau_{k} , t \right)} \mu \left( J \left( r \right) \right) \, \dd r} \right).
\end{align*}
Note that $M \left( t \right)$ is a c\`{a}dl\`{a}g stochastic process.
Because each $Z_k$ is strictly positive with probability $1$, it follows that
\begin{align*}
M \left( t \right) &\overset{\mathrm{d}}{=} \sum_{k=1}^{\infty} \left( \ind{\tau_{k} \leq t} - \ind{Z_{k} < \int_{t \wedge \tau_{k}}^{t} \mu \left( J \left( r \right) \right) \, \dd r} \right)\\
&= \sum_{k=1}^{\infty} \left( \ind{\tau_{k} \leq t} - \ind{\overline{Z}_{k} < \kappa \left( J \left( \tau_{k} \right) \right) \int_{t \wedge \tau_{k}}^{t} \mu \left( J \left( r \right) \right) \, \dd r} \right).
\end{align*}
If $J$ is deterministic, then it is relatively easy to determine the distribution of $M \left( t \right)$. For instance, one may compute the characteristic function of $M \left( t \right)$ via the following steps.

Suppose that $J \left( \omega , t \right) = f \left( t \right)$ for all $\omega \in \Omega$ and $t \geq 0$ for some function $f \in D \left( \left[ 0, \infty \right) ; \mathcal{E} \right)$. For fixed $\kappa$, $\mu$, $f$ and $t$ we define the functions $g$ and $h$ via
\begin{align*}
g \left( s \right) = \kappa \left( f \left( s \right) \right) \int_{t \wedge s}^{t} \mu \left( f \left( r \right) \right) \, \dd r,
\qquad
h \left( s \right) = 1 + \left[ \exp \left( \mathrm{i} \theta \right) - 1 \right] \exp \left( - g \left( s \right) \right).
\end{align*}
Now we may write the characteristic function of $M \left( t \right)$ as
\begin{align*}
&\mathbb{E} \exp \left( \mathrm{i} \theta M \left( t \right) \right) = \mathbb{E} \exp \left( \mathrm{i} \theta \sum_{k=1}^{\infty} \left( \ind{\tau_{k} \leq t} - \ind{\overline{Z}_{k} < \kappa \left( f \left( \tau_{k} \right) \right) \int_{t \wedge \tau_{k}}^{t} \mu \left( f \left( r \right) \right) \, \dd r} \right) \right) = \\
&= \mathbb{E} \ind{\tau_{1} > t} + \sum_{n=1}^{\infty} \mathbb{E} \ind{\tau_{n} \leq t ; \tau_{n+1} > t} \exp \left( \mathrm{i} \theta \left( n - \sum_{k=1}^{n} \ind{\overline{Z}_{k} < g \left( \tau_{k} \right)} \right) \right).
\end{align*}
Clearly, $\mathbb{E} \ind{\tau_{1} > t} = e^{- \int_{0}^{t} \lambda \left( f \left( s \right) \right) \, \dd s} = e^{-\Lambda \left( t \right)}$. We are left with computing the infinite sum above. Fix $n \in \mathbb{N}$ and note that
\begin{align*}
&\mathbb{E} \ind{\tau_{n} \leq t ; \tau_{n+1} > t} \exp \left( \mathrm{i} \theta \left( n - \sum_{k=1}^{n} \ind{\overline{Z}_{k} < g \left( \tau_{k} \right)} \right) \right) =\\
&= \mathbb{E} \left( \ind{\tau_{n} \leq t ; \tau_{n+1} > t} \condexpectation{\exp \left( \mathrm{i} \theta \left( n - \sum_{k=1}^{n} \ind{\overline{Z}_{k} < g \left( \tau_{k} \right)} \right) \right)}{\tau_{1} , \tau_{2} , \ldots} \right)\\
&= \mathbb{E} \ind{\tau_{n} \leq t ; \tau_{n+1} > t} \prod_{k=1}^{n} \left( 1 + \left[ \exp \left( \mathrm{i} \theta \right) - 1 \right] e^{- g \left( \tau_{k} \right)} \right),
\end{align*}
because $Y$ and $\overline{Z}_{1}, \overline{Z}_{2}, \ldots$ are independent.

Next, observe that
\begin{align*}
\mathbb{E} \ind{\tau_{n} \leq t ; \tau_{n+1} > t} \prod_{k=1}^{n} \left( 1 + \left[ \exp \left( \mathrm{i} \theta \right) - 1 \right] e^{- g \left( \tau_{k} \right)} \right) = 
\mathbb{E} \ind{\tau_{n} \leq t ; \tau_{n+1} > t} \prod_{k=1}^{n} h \left( \tau_{k} \right) \\
\end{align*}
and
\begin{align*}
\mathbb{E} \ind{\tau_{n} \leq t ; \tau_{n+1} > t} \prod_{k=1}^{n} h \left( \tau_{k} \right)
&= \mathbb{E} \left( \ind{\tau_{n} \leq t} \left( \prod_{k=1}^{n} h \left( \tau_{k} \right) \right) \condexpectation{\ind{\sigma_{n+1} > t - \tau_{n}}}{\tau_{1} , \ldots , \tau_{n}} \right) \\
&= \mathbb{E} \left( \ind{\tau_{n} \leq t} \left( \prod_{k=1}^{n} h \left( \tau_{k} \right) \right) e^{- \left( \Lambda \left( t \right) - \Lambda \left( \tau_{n} \right) \right)} \right).
\end{align*}
For convenience we write $x^{+}_{k} = x_{1} + \dotsb + x_{k}$. We have
\begin{align*}
&\mathbb{E} \left( \ind{\tau_{n} \leq t} \left( \prod_{k=1}^{n}  h \left( \tau_{k} \right) \right) e^{\Lambda \left( \tau_{n} \right)} \right) =\\
&= \mathbb{E} \ind{\overline{\tau}_{n} \leq \Lambda \left( t \right)} \left( \prod_{k=1}^{n} h \left( \Lambda^{-} \left( \overline{\tau}_{k} \right) \right) \right) e^{\overline{\tau}_{n}}\\
&= \int_{x_{1} = 0}^{\Lambda \left( t \right)} \int_{x_{2} = 0}^{\Lambda \left( t \right) - x^{+}_{1}} \dotsi \int_{x_{n} = 0}^{\Lambda \left( t \right) - x^{+}_{n-1}} \prod_{k=1}^{n} h \left( \Lambda^{-} \left( x^{+}_{k} \right) \right) \, \dd x_{n} \dotsc \dd x_{1}\\
&= \int_{y_{1} = 0}^{\Lambda \left( t \right)} \int_{y_{2} = y_{1}}^{\Lambda \left( t \right)} \dotsi \int_{y_{n} = y_{n-1}}^{\Lambda \left( t \right)} \prod_{k=1}^{n} h \left( \Lambda^{-} \left( y_{k} \right) \right) \, \dd y_{n} \dotsc \dd y_{1}\\
&= \int_{z_{1} = 0}^{t} \int_{z_{2} = z_{1}}^{t} \dotsi \int_{z_{n} = z_{n-1}}^{t} \prod_{k=1}^{n} \left[ h \left( z_{k} \right) \lambda \left( f \left( z_{k} \right) \right) \right] \, \dd z_{n} \dotsc \dd z_{1}.
\end{align*}
Now note that for an integrable function $g$ we have
\begin{align*}
\left[ \int_{0}^{t} g \left( s \right) \, \dd s \right]^{n} = n! \int_{z_{1} = 0}^{t} \int_{z_{2} = z_{1}}^{t} \dotsi \int_{z_{n} = z_{n-1}}^{t} \prod_{k=1}^{n} g \left( z_{k} \right) \, \dd z_{n} \dotsc \dd z_{1}.
\end{align*}
As a result, it holds that
\begin{align*}
&\int_{z_{1} = 0}^{t} \int_{z_{2} = z_{1}}^{t} \dotsi \int_{z_{n} = z_{n-1}}^{t} \prod_{k=1}^{n} \left[ h \left( z_{k} \right) \lambda \left( f \left( z_{k} \right) \right) \right] \, \dd z_{n} \dotsc \dd z_{1} =\\
&= \frac{1}{n!} \left[ \int_{0}^{t} h \left( s \right) \lambda \left( f \left( s \right) \right) \, \dd s \right]^{n}\\
&= \sum_{k=0}^{n} \frac{1}{k!} \Lambda \left( t \right)^{k} \frac{1}{\left( n-k \right)!} \left( \left[ \exp \left( \mathrm{i} \theta \right) - 1 \right] \int_{0}^{t} \lambda \left( f \left( s \right) \right) e^{- g \left( s \right)} \, \dd s \right)^{n-k}.
\end{align*}
Now we may write
\begin{align*}
&\mathbb{E} \exp \left( \mathrm{i} \theta M \left( t \right) \right) =\\
&= \mathbb{E} \ind{\tau_{1} > t} + \sum_{n=1}^{\infty} \mathbb{E} \ind{\tau_{n} \leq t ; \tau_{n+1} > t} \exp \left( \mathrm{i} \theta \left( n - \sum_{k=1}^{n} \ind{\overline{Z}_{k} < g \left( \tau_{k} \right)} \right) \right) \\
&= e^{- \Lambda \left( t \right)} + \sum_{n=1}^{\infty} e^{- \Lambda \left( t \right)} \sum_{k=0}^{n} \frac{1}{k!} \Lambda \left( t \right)^{k} \frac{1}{\left( n-k \right)!} \left( \left[ e^{\mathrm{i} \theta} - 1 \right] \int_{0}^{t} \lambda \left( f \left( s \right) \right) e^{- g \left( s \right)} \, \dd s \right)^{n-k}\\
&= e^{- \Lambda \left( t \right)} \sum_{n=0}^{\infty} \sum_{k=0}^{n} \frac{1}{k!} \Lambda \left( t \right)^{k} \frac{1}{\left( n-k \right)!} \left( \left[ e^{\mathrm{i} \theta} - 1 \right] \int_{0}^{t} \lambda \left( f \left( s \right) \right) e^{- g \left( s \right)} \, \dd s \right)^{n-k}\\
&= e^{- \Lambda \left( t \right)} \sum_{k=0}^{\infty} \sum_{n=0}^{\infty} \frac{1}{k!} \Lambda \left( t \right)^{k} \frac{1}{n!} \left( \left[ e^{\mathrm{i} \theta} - 1 \right] \int_{0}^{t} \lambda \left( f \left( s \right) \right) e^{- g \left( s \right)} \, \dd s \right)^{n}\\
&= \exp \left( \left[ \exp \left( \mathrm{i} \theta \right) - 1 \right] \int_{0}^{t} \lambda \left( f \left( s \right) \right) e^{- \kappa \left( f \left( s \right) \right) \int_{s}^{t} \mu \left( f \left( r \right) \right) \, \dd r} \, \dd s \right).
\end{align*}
Hence, in this case $M \left( t \right)$ has a Poisson distribution with parameter $\phi_{t} \left( f \right)$, where $\phi_{t} \left( f \right) = \int_{0}^{t} \lambda \left( f \left( s \right) \right) e^{- \kappa \left( f \left( s \right) \right) \int_{s}^{t} \mu \left( f \left( r \right) \right) \, \dd r} \, \dd s$. 

Now suppose that $J$ is not deterministic. Then $J$ is a random element of $D \left( \left[ 0, \infty \right) ; \mathcal{E} \right)$. In this case, we may use the independence of $J$ and standard arguments to obtain that
\begin{align*}
\mathbb{E} \exp \left( \mathrm{i} \theta M \left( t \right) \right)
&= \mathbb{E} \mathbb{E} \left[ \exp \left( \mathrm{i} \theta M \left( t \right) \right) \middle\vert \mathcal{F}^{J}_{\infty} \right]\\
&= \mathbb{E} \exp \left( \left[ \exp \left( \mathrm{i} \theta \right) - 1 \right] \int_{0}^{t} \lambda \left( J \left( s \right) \right) e^{- \kappa \left( J \left( s \right) \right) \int_{s}^{t} \mu \left( J \left( r \right) \right) \, \dd r} \, \dd s \right).
\end{align*}
We summarize our findings in the following lemma.
\begin{lemma}
Under the stated conditions, $M \left( t \right)$ has a Poisson distribution with random parameter
\begin{align*}
\phi_{t} \left( J \right) &= \int_{0}^{t} \lambda \left( J \left( s \right) \right) e^{- \kappa \left( J \left( s \right) \right) \int_{s}^{t} \mu \left( J \left( r \right) \right) \, \dd r} \, \dd s.
\end{align*}
\end{lemma}
Consequently, if we scale $\lambda \left( x \right) \mapsto n \lambda \left( x \right)$ and $J \mapsto J_{n}$, then the number of jobs in the system $M_{n} \left( t \right)$ has a Poisson distribution with random parameter $n \phi_{t} \left( J_{n} \right)$. This observation is crucial for the proof of the LDP for $\frac{1}{n} M_{n} \left( t \right)$.

\section{Continuity in Skorokhod space}
\label{sec:stepskor}
Let $\mathcal{E}$ be a metric space with metric $\rho$. Let $D \left( \left[ 0,\infty \right) ; \mathcal{E} \right)$ denote the space of c\`{a}dl\`{a}g functions $f \colon \left[ 0,\infty \right) \to \mathcal{E}$, i.e., $\lim_{s \downarrow t} f \left( s \right) = f \left( t \right)$ and $\lim_{s \uparrow t} f \left( s \right)$ exists in $\mathcal{E}$ for every $t \geq 0$, where $\lim_{s \uparrow 0} f \left( s \right) := f \left( 0 \right)$ by convention.

Define a metric $d^{\circ}$ on $D \left( \left[ 0,\infty \right) ; \mathcal{E} \right)$ via
\begin{align*}
d^{\circ} \left( f,g \right) = \inf_{\lambda \in \Lambda} \left[ \gamma \left( \lambda \right) \vee \int_{0}^{\infty} e^{-u} d \left( f , g , \lambda , u \right) \, \dd u \right].
\end{align*}
Here, $\Lambda$ denotes the space of increasing homeomorphisms of $\left[ 0,\infty \right)$,
\begin{align*}
\gamma \left( \lambda \right) = \sup_{t > s \geq 0} \left\vert \log \left( \lambda \left( t \right) - \lambda \left( s \right) \right) - \log \left( t - s \right) \right\vert
\end{align*}
and
\begin{align*}
d \left( f , g , \lambda , u \right) = \sup_{t \in \left[ 0,\infty \right)} \left[ 1 \, \wedge \, \rho \left( f \left( t \wedge u \right) , g \left( \lambda \left( t \right) \wedge u \right) \right) \right].
\end{align*}
The metric $d^{\circ}$ induces the Skorokhod $J_{1}$ topology. For more details, see \cite{ek1986} or \cite{whitt2002}.

\begin{definition}
A function $f_{\mathrm{c}} \in D \left( \left[ 0,\infty \right) ; \mathcal{E} \right)$ is called a \emph{piecewise constant function} or a \emph{step function} if there exist $n \in \mathbb{N}$, finitely many time points $0 = t_{0} < t_{1} < \ldots < t_{n} < \infty$ and $\alpha_{0} , \ldots , \alpha_{n} \in \mathcal{E}$ such that $f_{\mathrm{c}} \left( t \right) = \alpha_{i}$ for $t \in \left[ t_{i} , t_{i+1} \right)$ and $i = 0 , \ldots , n-1$ and $f_{\mathrm{c}} \left( t \right) = \alpha_{n}$ for $t \in \left[ t_{n} , \infty \right)$.

The set of step functions in $D \left( \left[ 0,\infty \right) ; \mathcal{E} \right)$ is denoted by $\mathcal{S} \left( \left[ 0,\infty \right) ; \mathcal{E} \right)$.
\end{definition}

\begin{proposition}
\label{pr:stepfunceps}
Let $f \in D \left( \left[ 0,\infty \right) ; \mathcal{E} \right)$. For all $T > 0$ and $\epsilon > 0$ there exists a step function $f_{\mathrm{c}} \in \mathcal{S} \left( \left[ 0,\infty \right) ; \mathcal{E} \right)$ such that
\begin{align*}
\sup_{t \in \left[ 0,T \right]} \rho \left( f \left( t \right) , f_{\mathrm{c}} \left( t \right) \right) < \epsilon.
\end{align*}
\end{proposition}
\begin{proof}
This is derived in the same way as \cite[Th.~12.2.2]{whitt2002}.
\end{proof}
\begin{corollary}
\label{co:stepdense}
The set $\mathcal{S} \left( \left[ 0,\infty \right) ; \mathcal{E} \right)$ is dense in $D \left( \left[ 0,\infty \right) ; \mathcal{E} \right)$.
\end{corollary}
Consequently, every continuous function on $D \left( \left[ 0,\infty \right) ; \mathcal{E} \right)$ is completely determined by its behavior on the set of step functions.

Let $\lambda \colon \mathcal{E} \to \left[ 0,\infty \right)$, $\kappa \colon \mathcal{E} \to \left[ 0,\infty \right)$ and $\mu \colon \mathcal{E} \to \left[ 0,\infty \right)$ be continuous. For $t \geq 0$, we would like to show that the function $\phi_{t} \colon D \left( \left[ 0,\infty \right) ; \mathcal{E} \right) \to \left[ 0,\infty \right)$ defined via
\begin{align}
\label{eq:phidef}
\phi_{t} \left( f \right) = \int_{0}^{t} \lambda \left( f \left( s \right) \right) e^{- \kappa \left( f \left( s \right) \right) \int_{s}^{t} \mu \left( f \left( r \right) \right) \, \dd r} \, \dd s
\end{align}
is a continuous function.

First, we observe that the map $c_{\lambda} \colon D \left( \left[ 0,\infty \right) ; \mathcal{E} \right) \to D \left( \left[ 0,\infty \right) ; \mathbb{R} \right)$ defined via $c_{\lambda} \left( f \right) \left( t \right) = \lambda \left( f \left( t \right) \right)$ is continuous, because $\lambda$ is continuous. Similarly, the functions $c_{\kappa}$ and $c_{\mu}$ are continuous.

Next, let $f,g \in D \left( \left[ 0,\infty \right) ; \mathbb{R} \right)$. Then pointwise multiplication of $f$ and $g$, defined via $\left( fg \right) \left( t \right) = f \left( t \right) g \left( t \right)$. This is a measurable map which is continuous at $\left( f,g \right)$ if $f$ or $g$ is continuous (cf.\ \cite[Th.~4.2]{whitt1980}).

Finally, let $f \in D \left( \left[ 0,\infty \right) ; \mathbb{R} \right)$. Then the map $\psi \colon D \left( \left[ 0,\infty \right) ; \mathbb{R} \right) \to D \left( \left[ 0,\infty \right) ; \mathbb{R} \right)$ defined via $\psi \left( t \right) = \int_{0}^{t} f \left( s \right) \, \dd s$ is continuous. This follows almost immediately from the definition of $\psi$ and the characterization in \cite[Pr.~3.5.3]{ek1986}.

Now note that the sequence of functions $\left\lbrace \lambda \left( f_{n} \right) \right\rbrace_{n \in \mathbb{N}}$ is bounded in the sup norm over $\left[ 0,t \right]$ if $f_{n} \to f$ in $D \left( \left[ 0,\infty \right) ; \mathcal{E} \right)$. Hence, it suffices to show that
\begin{align*}
\int_{0}^{t} e^{- \kappa \left( f_{n} \left( s \right) \right) \int_{s}^{t} \mu \left( f_{n} \left( r \right) \right) \, \dd r} \, \dd s \to \int_{0}^{t} e^{- \kappa \left( f \left( s \right) \right) \int_{s}^{t} \mu \left( f \left( r \right) \right) \, \dd r} \, \dd s
\end{align*}
as $f_{n} \to f$ in $D \left( \left[ 0,\infty \right) ; \mathcal{E} \right)$. But this follows from repeated applications of the first three observations.

Hence, the map $\phi_{t}$ must be continuous. Note that continuity of $\lambda$, $\kappa$ and $\mu$ is crucial to obtain this result. We summarize these findings in the following lemma.
\begin{lemma}
\label{lem:phictu}
Let $\lambda \colon \mathcal{E} \to \left[ 0,\infty \right)$, $\kappa \colon \mathcal{E} \to \left[ 0,\infty \right)$ and $\mu \colon \mathcal{E} \to \left[ 0,\infty \right)$ be continuous. Then the function $\phi_{t} \colon D \left( \left[ 0,\infty \right) ; \mathcal{E} \right) \to \left[ 0,\infty \right)$ as defined in equation \eqref{eq:phidef} is continuous.
\end{lemma}

\section{Properties of Poisson random variables}

For $\gamma \geq 0$, let $P_{0} \left( \gamma \right), P_{1} \left( \gamma \right), P_{2} \left( \gamma \right), \ldots$ denote a sequence of i.i.d.\ random variables that have a Poisson distribution with parameter $\gamma$. In this section, we will fix an arbitrary $x \in \mathbb{R}$, $\delta > 0$, $\lambda \geq 0$ and $\epsilon > 0$ and define $\lambda^{-}_{\epsilon} = \max \left\lbrace 0 , \lambda - \epsilon \right\rbrace$ and $\lambda^{+}_{\epsilon} = \lambda + \epsilon$. Recall that $\openballplus{\lambda}{\epsilon} = \openball{\lambda}{\epsilon} \cap \mathbb{R}_{+}$.

We would like to prove a large deviations lower bound for
\begin{align*}
\liminf_{n \to \infty} \inf_{\gamma \in \openballplus{\lambda}{\epsilon}} \frac{1}{n} \log \mathbb{P} \left( \frac{1}{n} \sum_{i=1}^{n} P_{i} \left( \gamma \right) \in \openball{x}{\delta} \right).
\end{align*}
Of course, the difficulty here is the presence of the infimum over a range of parameters. We will show in Proposition~\ref{pr:restrictedinf} that this infimum may be taken over certain restricted subsets of $\openballplus{\lambda}{\epsilon}$. For each of these subsets we will provide a large deviations lower bound, from which we will derive a lower bound when the infimum is taken over $\openballplus{\lambda}{\epsilon}$. This is the content of Proposition~\ref{pr:largedevopenint}.
\begin{proposition}
\label{pr:restrictedinf}
For all $x \in \mathbb{R}$, $\delta > 0$, $\lambda \geq 0$ and $\epsilon > 0$ it holds that
\begin{align*}
& \inf_{\gamma \in \openballplus{\lambda}{\epsilon}} \mathbb{P} \left( \frac{1}{n} \sum_{i=1}^{n} P_{i} \left( \gamma \right) \in \openball{x}{\delta} \right) = \\
& \displaystyle \inf_{\gamma \in \left( \openball{\lambda}{\epsilon} \cap \closedball{x}{\delta} \right) \cup \left\lbrace \lambda^{-}_{\epsilon} , \lambda^{+}_{\epsilon} \right\rbrace} \mathbb{P} \left( \frac{1}{n} \sum_{i=1}^{n} P_{i} \left( \gamma \right) \in \openball{x}{\delta} \right).
\end{align*}
\end{proposition}
\begin{proof}
Let $0 \leq \gamma_{-} \leq \gamma_{+} < \infty$. For $y \in \mathbb{R}$ it holds that
\begin{align}
\mathbb{P} \left( P_{0} \left( \gamma_{+} \right) = y \right) \geq \mathbb{P} \left( P_{0} \left( \gamma_{-} \right) = y \right) & & \text{ if } & & y \geq \gamma_{+} \geq \gamma_{-} \label{eq:gammaplusineq}
\intertext{and}
\mathbb{P} \left( P_{0} \left( \gamma_{+} \right) = y \right) \leq \mathbb{P} \left( P_{0} \left( \gamma_{-} \right) = y \right) & & \text{ if } & & \gamma_{+} \geq \gamma_{-} \geq y. \label{eq:gammaminineq}
\end{align}
Because we are working with i.i.d.\ Poisson random variables, we may write
\begin{align}
\mathbb{P} \left( \frac{1}{n} \sum_{i=1}^{n} P_{i} \left( \gamma \right) \in \openball{x}{\delta} \right) = \mathbb{P} \left( P_{0} \left( n \gamma \right) \in \left( n \left( x - \delta \right) , n \left( x + \delta \right)  \right) \right). \label{eq:sumofpoisson}
\end{align}
Now the statement of the proposition is an easy consequence of the equations \eqref{eq:gammaplusineq}, \eqref{eq:gammaminineq} and \eqref{eq:sumofpoisson} combined.
\end{proof}
\begin{proposition}
\label{pr:limxdeltazero}
Let $x \in \mathbb{R}$ and $\delta > 0$. If $\openballplus{x}{\delta} \not= \emptyset$, then
\begin{align*}
\lim_{n \to \infty} \inf_{\gamma \in \closedballplus{x}{\delta}} \frac{1}{n} \log \mathbb{P} \left( \frac{1}{n} \sum_{i=1}^{n} P_{i} \left( \gamma \right) \in \openball{x}{\delta} \right) = 0.
\end{align*}
\end{proposition}
\begin{proof}
For a Borel set $A \subset \mathbb{R}$, define $p_{n} \left( A \middle\vert \gamma \right) = \mathbb{P} \left( \frac{1}{n} \sum_{i=1}^{n} P_{i} \left( \gamma \right) \in A \right)$. Now suppose that $\openballplus{x}{\delta} \not= \emptyset$. Then the diameter of $\openballplus{x}{\delta}$ is strictly positive and bounded above by $r = \min \left\lbrace 2 \delta , x + \delta \right\rbrace$.

Let $N_{r} \in \mathbb{N}$ be such that $\tfrac{1}{N_{r}} < \tfrac{r}{2}$. Then for all $n \geq N_{r}$ and $\gamma \in \closedballplus{x}{\delta}$ we define $\gamma^{-}_{n} = \tfrac{1}{n} \left\lfloor n \gamma \right\rfloor$, $\gamma^{+}_{n} = \tfrac{1}{n} \left\lceil n \gamma \right\rceil$ and
\begin{align*}
\gamma^{*}_{n} = \min \left\lbrace \left\lbrace \gamma^{-}_{n} , \gamma^{+}_{n} \right\rbrace \cap \openball{x}{\delta} \right\rbrace.
\end{align*}
Then $\max \left\lbrace \abs{\gamma - \gamma^{-}_{n}} , \abs{\gamma - \gamma^{+}_{n}} \right\rbrace \leq \tfrac{1}{n} < \tfrac{r}{2}$ and $p_{n} \left( \openball{x}{\delta} \middle\vert \gamma \right) \geq p_{n} \left( \left\lbrace \gamma^{*}_{n} \right\rbrace \middle\vert \gamma \right)$ for each $n \in \mathbb{N}$ and each $\gamma \in \closedballplus{x}{\delta}$. Using that $n! \leq n^{n + 1/2} e^{-n+1}$, we get
\begin{align*}
p_{n} \left( \left\lbrace \gamma^{*}_{n} \right\rbrace \middle\vert \gamma \right)
&\geq \left(  \frac{n \gamma}{n \gamma + 1}  \right)^{n \gamma^{*}_{n}} e^{n \left( \gamma^{*}_{n} - \gamma \right)} e^{-1} \left( n \gamma^{*}_{n} \right)^{-1/2}\\
&\geq \left( 1 - \frac{1}{n \left( x + \delta \right) + 1} \right)^{n \left( x + \delta \right)} e^{-2} \left( n \left( x + \delta \right) \right)^{-1/2}
\end{align*}
for each $n \in \mathbb{N}$ and each $\gamma \in \closedballplus{x}{\delta}$. This implies the statement.
\end{proof}
Combined with Cram\'{e}r's~Theorem in $\mathbb{R}$, the two previous propositions enable us to prove the following large deviations bound. Note that we prove an equality rather than an inequality and that the limit exists.
\begin{proposition}
\label{pr:largedevopenint}
For all $x \in \mathbb{R}$, $\delta > 0$, $\lambda \geq 0$ and $\epsilon > 0$ it holds that
\begin{align}
\label{eq:openintlowbnd}
\lim_{n \to \infty} \inf_{\gamma \in \openballplus{\lambda}{\epsilon}} \frac{1}{n} \log \mathbb{P} \left( \frac{1}{n} \sum_{i=1}^{n} P_{i} \left( \gamma \right) \in \openball{x}{\delta} \right) = \min_{\gamma \in \left\lbrace \lambda^{-}_{\epsilon} , \lambda^{+}_{\epsilon} \right\rbrace} \left[ - \inf_{a \in \openball{x}{\delta}} \ell \left( \gamma ; a \right) \right].
\end{align}
\end{proposition}
\begin{proof}
Define $p_{n} \left( A \, \middle\vert \, \gamma \right) = \mathbb{P} \left( \frac{1}{n} \sum_{i=1}^{n} P_{i} \left( \gamma \right) \in A \right)$ for Borel sets $A \subset \mathbb{R}$ and $C = \left( \openball{\lambda}{\epsilon} \cap \closedball{x}{\delta} \right) \cup \left\lbrace \lambda^{-}_{\epsilon} , \lambda^{+}_{\epsilon} \right\rbrace$. 
Thanks to Proposition~\ref{pr:restrictedinf} we may write
\begin{align*}
\lim_{n \to \infty} \inf_{\gamma \in \openballplus{\lambda}{\epsilon}} \frac{1}{n} \log p_{n} \left( \openball{x}{\delta} \, \middle\vert  \, \gamma \right) &=
\lim_{n \to \infty} \displaystyle \inf_{\gamma \in C} \frac{1}{n} \log p_{n} \left( \openball{x}{\delta} \, \middle\vert \, \gamma \right).
\end{align*}
It follows from Proposition~\ref{pr:limxdeltazero} that we may restrict the infimum to the set $\left\lbrace \lambda^{-}_{\epsilon} , \lambda^{+}_{\epsilon} \right\rbrace$, so
\begin{align*}
\lim_{n \to \infty} \displaystyle \inf_{\gamma \in C} \frac{1}{n} \log  p_{n} \left( \openball{x}{\delta} \, \middle\vert  \, \gamma \right)
&= \lim_{n \to \infty} \displaystyle \min_{\gamma \in \left\lbrace \lambda^{-}_{\epsilon} , \lambda^{+}_{\epsilon} \right\rbrace} \frac{1}{n} \log  p_{n} \left( \openball{x}{\delta} \, \middle\vert  \, \gamma \right) \\
&= \displaystyle \min_{\gamma \in \left\lbrace \lambda^{-}_{\epsilon} , \lambda^{+}_{\epsilon} \right\rbrace} \lim_{n \to \infty} \frac{1}{n} \log  p_{n} \left( \openball{x}{\delta} \, \middle\vert  \, \gamma \right) \\
&= \min_{\gamma \in \left\lbrace \lambda^{-}_{\epsilon} , \lambda^{+}_{\epsilon} \right\rbrace} \left[ - \inf_{a \in \openball{x}{\delta}} \ell \left( \gamma ; a \right) \right].
\end{align*}
The last equality is an application of Cram\'{e}r's~Theorem for i.i.d.\ Poisson random variables; the limit exists because $\openball{x}{\delta}$ is a continuity set for the Fenchel-Legendre transform corresponding to a Poisson distribution.
\end{proof}

As shown in the inequalities \eqref{eq:ell_leq} and \eqref{eq:ell_geq}, the Fenchel-Legendre transforms corresponding to Poisson distributions are nicely ordered in some sense. This property leads to the following propositions. Their proofs are elementary but tedious and are therefore omitted.
\begin{proposition}
\label{pr:verifyvaradhan}
Let $F \subset \mathbb{R}$ be closed and define $f \colon \left[ 0,\infty \right) \to \left[ -\infty,0 \right]$ via
\begin{align*}
f \left( \gamma \right) = - \inf_{a \in F} \ell \left( \gamma;a \right).
\end{align*}
If $F \subset \left( -\infty,0\right)$, then $f \equiv -\infty$. If $F \cap \left[ 0,\infty \right) \not= \emptyset$, then $f$
is real-valued and continuous on $\left( 0,\infty \right)$. Additionally, $\lim_{\gamma \downarrow 0} f \left( \gamma \right) = f \left( 0 \right)$, where $f \left( 0 \right) = 0$ if $0 \in F$ and $f \left( 0 \right) = \infty$ if $0 \not\in F$. In any case, $f^{-1} \left( \left[ a,b \right] \right)$ is closed for all $a,b \in \left( -\infty, 0 \right]$ with $a \leq b$.
\end{proposition}
 
\begin{proposition}
\label{pr:Iratefunction}
Let $\mathcal{R} \subset \left[ 0,\infty \right)$ be a non-empty, closed set. Let $\psi \colon \mathbb{R} \to \left[ 0,\infty \right]$ be a lower semi-continuous function. Then the function $I \colon \mathbb{R} \to \left[ 0,\infty \right]$ defined via
\begin{align*}
I \left( a \right) = \inf_{\gamma \in \mathcal{R}} \left[ \ell \left( \gamma ; a \right) + \psi \left( \gamma \right) \right]
\end{align*}
is a lower semi-continuous function.
\end{proposition}

\bibliographystyle{plain}
\bibliography{ref}

\end{document}